\documentclass[english]{article}
\usepackage[T1]{fontenc}
\usepackage{amsmath,amsthm,amssymb, amsfonts}
\usepackage{comment}
\usepackage[left=1in, right=1in, top=1in, bottom=1in]{geometry} 
\usepackage{graphicx}
\usepackage{indentfirst}
\usepackage{bm}
\usepackage{mathtools}
\usepackage[colorlinks]{hyperref}
\usepackage{csquotes}
\usepackage[most]{tcolorbox} 
\usepackage[ backend=biber, style=alphabetic, sorting=nyt, giveninits=true, maxnames=99]{biblatex}
\usepackage{bookmark} 
\usepackage{aliascnt}

\newcommand{\W}{\mathcal{W}}
\newcommand{\R}{\mathbb{R}}
\newcommand{\Pro}{\mathcal{P}_2(\R^d)}
\newcommand{\V}{\mathcal{V}}
\newcommand{\U}{\mathcal{U}}
\newcommand{\id}{\text{id}}
\newcommand{\dr}{\mathrm{d}}

\numberwithin{equation}{section}

\newtheorem{theorem}{Theorem}[section]

\newtheorem{prop}[theorem]{Proposition}

\newaliascnt{lemma}{theorem}
\newtheorem{lemma}[lemma]{Lemma}
\aliascntresetthe{lemma}

\theoremstyle{definition}
\newaliascnt{definition}{theorem}
\newtheorem{definition}[definition]{Definition}
\aliascntresetthe{definition}

\theoremstyle{remark}
\newtheorem{remark}[theorem]{Remark}

\DeclareMathOperator*{\argmin}{arg\,min}

\title{The Mortensen observer on the space of probability measures}
\author{Martin Morange\thanks{CMAP, CNRS, INRIA, École polytechnique, Institut Polytechnique de Paris, Palaiseau, France}}
\date{}

\begin{document}

\maketitle

\begin{abstract}
We study a deterministic filtering problem formulated directly on the Wasserstein space of probability measures with finite second moment. Motivated by the Mortensen minimum-energy observer, we consider the reconstruction of an evolving probability density from partial observations by minimizing an action functional combining a kinetic transport cost and a time-dependent observation mismatch. The resulting value function is defined on the infinite-dimensional manifold $(\Pro, W_2)$ and satisfies a Hamilton-Jacobi-Bellman equation involving the Wasserstein gradient.

Under suitable regularity and growth assumptions on the observation functional, we establish dynamic programming principles, continuity of the value function, existence of minimizing trajectories, and viscosity solution properties of the associated Hamilton-Jacobi equation. We provide two complementary notions of viscosity solutions: a geometric formulation based on subdifferentials in Wasserstein space, and a Hilbertian formulation inspired by Lions' lifting approach. This allows us to prove a comparison principle and uniqueness of solutions. Extensions to transport equations with drift are also discussed. Finally, we introduce a semi-Lagrangian scheme in order to approximate the value function, and show $\Gamma$-convergence of the scheme. 
\end{abstract}

\tableofcontents

\section*{Introduction}

State estimation and filtering for dynamical systems subject to uncertainty is a fundamental problem in control theory, studied for instance in \cite{KRENER} or \cite{BENSOUSSAN}. In the deterministic setting, the Mortensen observer \cite{MORTENSEN} formulates filtering as a minimum-energy control problem, leading to a Hamilton-Jacobi-Bellman (HJB) equation governing the value function. While this framework is well understood for finite-dimensional systems, its extension to infinite-dimensional objects-such as probability measures-remains largely unexplored.
\medbreak
In many applications, the state of interest is naturally described by a probability density rather than a finite-dimensional vector. Examples arise in mean-field systems, collective dynamics or kinetic equations. In such settings, the natural configuration space is the Wasserstein space $\Pro$, endowed with the quadratic Wasserstein distance $\W_2$. This space possesses a rich geometric structure that allows one to interpret transport equations as absolutely continuous curves and to define generalized velocities and energies.
\medbreak
The geometric and analytical structure of the Wasserstein space $\Pro$ has been extensively developed over the past two decades, notably through the theory of optimal transport and gradient flows, thoroughly detailed in \cite{AMBROSIO}. The dynamic formulation of optimal transport introduced by Benamou and Brenier \cite{BENAMOU} provides a kinetic interpretation of the Wasserstein distance and underpins much of the modern control-theoretic perspective on measure-valued dynamics. A systematic presentation of these ideas can be found in the monographs \cite{AMBROSIO,VILLANI,OTAM}.
\medbreak
Hamilton-Jacobi equations on Wasserstein space first appeared implicitly in the study of optimal transport and mean-field limits. A seminal contribution established the connection between variational problems on $\Pro$ and Hamilton-Jacobi equations in \cite{GANGBO}, proving lower semi-continuity and subsolution properties of the associated value functions, and full viscosity solution results in dimension one. These results were later extended to higher dimensions and more general settings, notably in \cite{HYNDCOMP}.
\medbreak
In parallel, the theory of mean-field games and mean-field control revealed a deep link between Hamilton-Jacobi equations on spaces of probability measures and systems of interacting agents. From this viewpoint, Hamilton-Jacobi equations on Wasserstein space can be interpreted as infinite-population limits of classical deterministic control problems, see for instance \cite{JIMENEZ} or \cite{MARIGONDA}. We refer to the lectures \cite{LIONS} and to the systematic treatment in \cite{CardaliaguetDelarueLasryLions2019} for a detailed account of this perspective.
\medbreak
A major difficulty in this context lies in defining appropriate notions of differentiability and viscosity solutions in infinite-dimensional, non-Hilbertian spaces. Two main approaches have emerged. The first is intrinsic and geometric, relying on Wasserstein subdifferentials, metric slopes, and optimal couplings, as developed for instance in \cite{GANGBO,HYNDCOMP}, or by introducing "smooth" test functions, as in \cite{JEAN, AUSSEDAT}. The second approach, pioneered in \cite{LIONS}, consists in lifting functions defined on $\Pro$ to suitable Hilbert spaces of square-integrable random variables, where classical viscosity solution theory applies. This Hilbertian framework has recently led to robust comparison principles and well-posedness results for Hamilton-Jacobi equations on Wasserstein spaces, for instance in \cite{BERTUCCI, TUDORASCU} or in \cite{DAUDIN, SEEGER} for second-order HJB equations.
\medbreak
Despite these advances, applications to filtering and observer design remain scarce. Classical filtering theory, whether stochastic or deterministic, has been largely confined to finite-dimensional systems or function-valued states, for instance in \cite{BREITEN} for a study of the Mortensen observer in finite dimension or \cite{GAEL} for its generalization to Riemannian manifolds. In the Euclidean case, advances were made for the estimation of partial differential equations, both in the linear case in \cite{BENSOUSSAN} or in the nonlinear reversible case in \cite{SCHRODER}. The present work fills this gap by providing a systematic Mortensen-type filtering framework on the Wasserstein space, combining optimal transport techniques, dynamic programming, and viscosity solution theory.
\medbreak
In this work, we propose a deterministic filtering framework directly posed on Wasserstein space. We consider densities evolving under a continuity equation driven by a control velocity field and penalize both the kinetic transport cost, measured in the spirit of the Benamou-Brenier formulation of optimal transport, and a time-dependent observation mismatch functional. The resulting value function is defined on $\mathbb R_+ \times \Pro$ and formally satisfies a HJB equation involving the squared $L^2(\mu)$-norm of the Wasserstein gradient.
\medbreak
Our first contribution is to rigorously establish the well-posedness of this value function: we prove a dynamic programming principle, continuity with respect to time and the Wasserstein topology, and existence of minimizing trajectories. We then show that the value function is a viscosity solution to the associated Hamilton-Jacobi equation on $\Pro$. Two complementary approaches are developed. The first relies on the intrinsic differential structure of Wasserstein space and builds on earlier work on viscosity solutions for optimal transport problems. The second uses a Hilbertian lifting technique inspired by Lions, which allows us to prove a comparison principle and uniqueness under suitable Lipschitz assumptions.
\medbreak
Finally, we show how the framework extends naturally to transport equations with drift. We also provide a semi-Lagrangian scheme in order to approximate the value function, which is based on a scheme designed in \cite{MOIREAU} and \cite{KRENER} for the finite dimensional case, and are able to show $\Gamma$-convergence of the scheme towards the value function. Beyond filtering, our results contribute to the broader theory of Hamilton-Jacobi equations on spaces of probability measures and reinforce the deep connections between optimal control, mean-field systems, and optimal transport.

\medbreak

\textbf{Organization of the paper.} The rest of the paper is organized as follows. In Section 1, we provide some background on both the Mortensen observer in finite dimensional Euclidean spaces as well as on classical results in optimal transport. In Section 2, we introduce the optimal filtering problem on the space of probability measures and its relation with a value function. We then show in Section 3 several properties of this value function, by showing continuity of this function as well as viscosity properties it satisfies. In Section 4, we extend those results to the case of a transport equation. We then provide an approximation scheme for this function in Section 5, based on a semi-Lagrangian approach. Finally, in Section 6, we show how our approach allows us to generalize the classical deterministic filtering problem by linking the value function with the solution to a finite dimensional HJB equation.

\section{Background on the Mortensen observer in finite dimension and on optimal transport}

\subsection{The Mortensen observer in finite dimension} \label{MORTENSENCLASSIC}

In this section, we recall the key ideas surrounding non-linear deterministic filtering from a control theoretic point of view. We rely on \cite{BARAS} and \cite{FLEMING}, which present the Mortensen observer (also called the minimum energy estimator) thoroughly, an observer that was first designed in \cite{MORTENSEN}. We also refer to \cite{LPC}, \cite{GAEL} and \cite{BREITEN} for recent developments on the topic. In the following, we consider a perturbed ODE on $\R^d$, taking the form of 

\begin{equation*}
    \begin{dcases}
    \dot{x}(t) = b(x(t)) +  v(t), \: t\geq 0, \\
    x(0) = x_0 + w,
    \end{dcases}
\end{equation*}

Here, $v$ and $w$ should be interpreted as some error on the initial condition of $x$ and on the model governing the dynamics of $x$. We further assume that we have at our disposal a set of observations, taking the form

\begin{equation*}
    y(t) = h(x(t)) + \eta(t).
\end{equation*}

The map $h$ is called the observation map, from $\R^d$ to $\R^m$. In the following, we assume 
\begin{itemize}
    \item the map $b$ is in $C^1(\R^d)$ with $D_x b$ uniformly bounded and
        \begin{equation*}
            |b(x)| \leq C(1 + |x|),
        \end{equation*}
        \item the map $h$ is in $C^1(\R^d, \R^m)$ with $D_x h$ uniformly bounded and
        \begin{equation*}
            |h(x)| \leq C'(1 + |x|),
        \end{equation*}
    \end{itemize}

so that the Cauchy-Lipschitz theorem ensures global existence and uniqueness of solutions to this system in $E= H^1(\R^+, \R^d) \cap C(\R^+, \R^d)$. In the following, we will often write $\bm{x}$ for a curve $x \in E$ in order to avoid any confusion with an $x \in \R^d$. Our goal is to reconstruct a complete trajectory $x$, from the observations $\{ y(t), t \in \R_+ \}$. One approach, often referred to as optimal, is to consider a least-square estimator, by minimizing the criterion

\begin{equation*}
    \mathcal{J}_t(\bm{x}) = V_0(x(0)) + \int_0^t \frac{1}{2} |\dot{x}(s) - b(x(s))|^2 + \frac{1}{2} | y(s) - h(x(s))|^2 \, \mathrm{d}s,
\end{equation*}

on every admissible trajectory $\bm{x} = (x(s))_{s \geq 0}$. Here, $V_0$ is strictly convex, is in $C^1(\R^d)$ and satisfies
        \begin{equation*}
            0 \leq V_0(x) \leq C''(1 + |x|^2),
        \end{equation*} 
        and acts as some sort of control on the initial condition of $\bm{x}$ (usually, one takes $V_0(x) = \frac{1}{2} |x - x_0|^2$ for some a priori $x_0$ on the initial condition). The other terms ensure that $\bm{x}$ remains close to the ideal dynamics $\dot{x}(t) = b(x(t))$ while keeping in mind that $h({x}(t))$ has to be close to $y$. 

\begin{definition}
The Mortensen estimator is defined as 

\begin{equation*}
    \hat{x}(t) = \bar{x}_t(t) \quad \text{for every } t\in \R_+,
\end{equation*}

where $\bm{\bar{x}}$ is the optimal trajectory minimizing $\mathcal{J}_t$.
\end{definition}

\begin{remark}
It is worth noting that this estimator generalizes the famous Kalman-Bucy filter, which was designed for linear ODEs with linear observations (i.e. $h$ linear) in \cite{KALMAN} or \cite{WILLEMS} for its deterministic interpretation. A proof of this result is presented exhaustively in \cite[Subsection 2.2]{MOIREAU} or in \cite{KUNISCH}.
\end{remark}

We define the value function associated to the minimization of $\mathcal{J}_t$ via

\begin{equation*}
    V(t,x) = \inf_{\substack{\bm{x} \in E \\ x(t) = x}} \mathcal{J}_t(\bm{x}).
\end{equation*}

One is then able to prove that $V$ is a viscosity solution --~in the sense of \cite{CRANDALL}~-- to a Hamilton-Jacobi-Bellman equation. The value function is particularly interesting because it allows one to define the Mortensen estimator. Indeed, we have that $\hat{x}(t) \in \argmin_{x \in \R^d} V(t,x)$ for every $t \in \R^+$. Even better, if $V$ is regular enough, the dynamics of $\hat{x}$ can be described explicitly. 
\begin{theorem}{\cite[Theorem 3.1]{BARAS}}
$V$ is a viscosity solution to

\begin{equation*} \label{HJBCLASSIC}
\begin{dcases}
    \partial_t V(t,x) + b(x) \cdot \nabla V(t,x) +\frac{1}{2} |\nabla V (t,x) |^2 - \frac{1}{2} | y(t) - h(x)|^2=0, \\
    V(0,x) = V_0(x) \quad \forall x\in \R^d.
\end{dcases}
\end{equation*}
\end{theorem}

It is important to note that the behaviour of the value function $V$ is well-understood -- at least in the sublinear case --. For instance, under the sublinear growth assumptions we gave at the beginning of this section, we know that there exists some $K>0$ such that

    \begin{equation} 
        V(t,x) \leq K(1 + |x|^2).
    \end{equation}

    The proof of this result can be found in \cite[Proposition 3.1]{DALIO} for instance. Under regularity assumptions on $V$, we are even capable of describing the dynamics of the Mortensen observer $\hat{x}(t)$.

\begin{theorem}{\cite[Remark 2.4]{FLEMING}}\label{DYNMORTENSEN}
Assume that $V(t,\cdot) \in C^2(\R^d)$ and that its Hessian is invertible around $\hat{x}(t)$ for every $t \in \R^+$. Then $\hat{x}$ is a solution to

\begin{equation*} \label{MORTENSENCLA}
   \begin{dcases}
    \dot{\hat{x}}(t) = b(\hat{x}(t)) - \left( \nabla^2 V (t, \hat{x}(t)) \right)^{-1} D_x h(\hat{x}(t))^T (y(t) - h(\hat{x}(t))), \\
    \hat{x}(0) \in \argmin V_0.
    \end{dcases}
\end{equation*}
\end{theorem}

\begin{remark}
The assumption that $V$ is $C^2$ with an invertible Hessian is often not satisfied, which renders \autoref{DYNMORTENSEN} not usable in a lot of cases. However, if $V_0$ has an invertible Hessian, it can be shown --~ at least in dimension $1$, by using the link between HJB and Burgers equations ~-- that the Hessian of $V$ stays invertible on a small time interval. Furthermore, it was recently proven in \cite{BREITEN} that the Hessian is globally in time invertible for small observation noises and for bilinear dynamics. However, it is more efficient to compute $V$ and then find its minimum, a problem for which a lot of literature is available.
\end{remark}

\subsection{Optimal transport and Wasserstein distance}

The optimal transport problem asks for the optimal way to
redistribute a mass distribution $\mu$ into a target distribution $\nu$,
where the cost of moving a unit of mass from $x$ to $y$ is prescribed by
a ground cost $c(x,y) \geq 0$.  Originally formulated in
\cite{MONGE} for the cost $c(x,y) = |x-y|$, the problem resisted a
complete analytical treatment for two centuries, owing in large part to the nature of transport maps, i.e. measurable maps
$T: X \to Y$ satisfying $T_{\#}\mu = \nu$, where the push-forward
$T_{\#}\mu$ is defined by $(T_{\#}\mu)(B) = \mu(T^{-1}(B))$ for every
Borel set $B$.  The Monge problem
\[
    \inf_{\substack{T:\, \R^d \to \R^d \\ T_{\#}\mu\,=\,\nu}}
    \int_{\R^d} c\bigl(x,\,T(x)\bigr)\,\mathrm{d}\mu(x),
\]
may fail to have a solution even in simple cases --~for instance, when $\mu$
is a Dirac mass and $\nu$ is not~-- because a single source point cannot be
split among several targets.
 
The decisive reformulation is due to \cite{KANTOROVICH},
which relaxed the problem by allowing mass to be split. Instead of a
transport map, one seeks a transport plan, i.e. a probability
measure $\gamma$ on $\R^d \times \R^d$ with marginals $\mu$ and $\nu$, and one
minimizes
\[
    \mathcal{C}(\mu,\nu)
    \;=\;
    \inf_{\gamma \in \Pi(\mu,\nu)}
    \int_{\R^d \times \R^d} c(x,y)\,\mathrm{d}\gamma(x,y),
\]
where $\Pi(\mu,\nu)$ denotes the set of all such couplings. In the following, we will denote by $\Pi_o$ the set of optimal couplings.

Specialising to the cost $c(x,y) = |x - y|^p$ for $p \geq 1$, one defines the $p$-Wasserstein distance
between $\mu, \nu \in \mathcal{P}_p(\R^d)$ by
\[
    W_p(\mu,\nu)
    \;=\;
    \left(
        \inf_{\gamma \in \Pi(\mu,\nu)}
        \int_{\R^d \times \R^d} |x-y|^p\,\mathrm{d}\gamma(x,y)
    \right)^{\!1/p},
\]
where $\mathcal{P}_p(\R^d)$ denotes the set of Borel probability measures on
$\R^d$ with finite $p$-th moment.
 
The case $p = 2$ is of particular geometric
importance.  Building on Brenier's theorem and on
\cite{MCCANN}, \cite{OTTO} showed
that $(\mathcal{P}_2(\mathbb{R}^d), W_2)$ carries a formal
infinite-dimensional Riemannian structure, where the tangent space at a measure $\mu$ is identified with
\[
    \mathrm{Tan}_\mu\,\mathcal{P}_2(\mathbb{R}^d)
    \;=\;
    \overline{\bigl\{\nabla\varphi :
    \varphi \in C^\infty_c(\mathbb{R}^d)\bigr\}}^{L^2(\mu;\,\mathbb{R}^d)},
\]
and the geodesic between $\mu$ and $\nu$ is given by McCann's
displacement interpolation
\[
    \mu_t
    \;=\;
    \bigl((1-t)\,\mathrm{id} + t\,\nabla\phi\bigr)_{\#}\mu,
    \qquad t \in [0,1],
\]
a constant-speed geodesic satisfying
$W_2(\mu_s,\mu_t) = |t-s|\,W_2(\mu,\nu)$.
 
A dynamical interpretation of $W_2$ was given in \cite{BENAMOU}, which showed that the squared Wasserstein
distance admits the representation
\begin{equation}
    \label{eq:BB}
    W_2(\mu,\nu)^2
    \;=\;
    \inf\left\{
        \int_0^1\!\int_{\mathbb{R}^d}
        |v(t,x)|^2\,\mathrm{d}\mu_t(x)\,\mathrm{d}t
        \;\;\middle|\;\;
        \partial_t\mu_t + \operatorname{div}(\mu_t v_t) = 0,\;
        \mu_0 = \mu,\;\mu_1 = \nu
    \right\},
\end{equation}
where the infimum ranges over all Borel velocity fields $v$ for which the
continuity equation holds in the distributional sense.  This formula is the
exact analogue of the length-minimization characterization of geodesics in
Riemannian geometry. 
  
This formula led to consider absolutely continuous curves on the space of probability measures. Recall that a curve $(\mu_t)_{t \in (a,b)}$ in a metric space
$(M,d)$ is absolutely continuous if there exists
$g \in L^1(a,b)$ with 

\begin{equation*}
d(\mu_s,\mu_t) \leq \int_s^t g(r)\,\mathrm{d}r,
\end{equation*}

for all $s \leq t$. The metric derivative
\begin{equation*}
|\dot{\mu}_t| = \lim_{h \to 0} d(\mu_{t+h},\mu_t)/|h|
\end{equation*}

then exists for a.e.\ $t$ and is the $L^1$-minimal such $g$. The class $\mathrm{AC}^2((a,b);\mathcal{P}_2(\mathbb{R}^d))$ consists of
those absolutely continuous curves for which
$|\dot{\mu}| \in L^2(a,b)$. In the following, we will denote by $\bm{\mu} = (\mu_t)_{t \in [a,b]}$ a curve on the Wasserstein space. The characterization theorem given in \cite{AMBROSIO} states that a
narrowly continuous curve
$(\mu_t) \subset \mathcal{P}_p(\mathbb{R}^d)$ belongs to
$\mathrm{AC}^p\bigl((a,b);\mathcal{P}_p(\mathbb{R}^d)\bigr)$ if and only
if there exists a Borel vector field
$v : (a,b) \times \mathbb{R}^d \to \mathbb{R}^d$ with
$\int_a^b\!\int_{\mathbb{R}^d}|v_t|^p\,\mathrm{d}\mu_t\,\mathrm{d}t
< \infty$ satisfying the continuity equation
\begin{equation*}
\partial_t\mu_t + \operatorname{div}(\mu_t v_t) = 0,
\end{equation*}
in the sense of distributions. Moreover, among all such velocity fields there exists a unique
minimal one, $v_t^{\min}$, belonging to the
$L^p(\mu_t)$-closure of gradients and realizing
$|\dot{\mu}_t|^p = \int_{\mathbb{R}^d}|v_t^{\min}|^p\,\mathrm{d}\mu_t$
for a.e.\ $t$.  

\section{The filtering problem on the space of probability measures}
In this work, we wish to extend the Mortensen observer to a more general setting, namely the --~at least intuitively~-- infinite-dimensional manifold $(\Pro, \mathcal{W}_2)$. We focus on a filtering problem associated with the model 

\begin{equation}
    \partial_t \rho(t,x) + \nabla \cdot (\rho (t,x)v(t,x)) =0,
    \label{SIMPLE}
\end{equation}

We assume that we have access to a certain time-dependent positive energy $\mathcal{E}(t,\rho)$ where $\mathcal{E} : \R_+ \times \Pro \rightarrow \R$. In the context of observations, one could have, for instance,

\begin{equation*}
    \mathcal{E}(t,\rho) = \frac{1}{2} \int_{\R^d} \left\|y(t) - h(t,x) \right\|_{\mathfrak{Y}}^2 \dr \rho, \quad  \mathcal{E}(t,\rho) = \frac{1}{2} \left|y(t) - \int_{\R^d} |x|^2 \,\mathrm{d}\rho \right|^2, \quad \mathcal{E}(t, \rho) = \frac{1}{2} \W_2^2(\rho, \rho^{\text{obs}}(t))
\end{equation*}

which echoes the formalism of the Mortensen observer, where we wish to tackle observations depending only on the law of a process, rather than having access to every realization. The three cases we displayed correspond to the observation of $h(t,\cdot)$ for every $x$ in the support of the measure, of the second moment of the law, and to a $\W_2$ distance between the measure and a target $\rho^{\text{obs}}$ respectively. While the first type of observation resembles the classical Mortensen setting, the two others are more "mean-field", in the sense that it corresponds to observing only the law of the underlying process. An example of such a setting was studied in \cite{DOUMIC}, where the authors proposed a Kalman-based observer approach with observations given by the second moment of the underlying probability density in order to reconstruct the density. We wish to extend this kind of approach to a larger class of systems and to derive a systematic observer design for this kind of problem in the spirit of the Mortensen observer presented in \autoref{MORTENSENCLASSIC}. 

\begin{remark}
    In fact, the first type of observation can be recovered when we set $\rho_{\text{obs}}(t) = \delta_{y(t)}$ and by considering a modified Wasserstein-type metric of the form

    \begin{equation*}
        \W_h^2(\mu, \nu) = \inf_{\pi \in \Pi(\mu, \nu)} \int_{\R^{d} \times \mathfrak{Y}} \| y -  h(t,x) \|_{\mathfrak{Y}}^2 \pi(\dr x, \dr y),
    \end{equation*}

    when $h$ is sublinear for instance. One has to keep in mind this example throughout the paper as a way of bridging the classical Mortensen observer and our approach, a link that will be investigated more carefully in \autoref{LINK}.
\end{remark}

Let us highlight the fact that this model is the "simplest" setting in $\Pro$, as it corresponds to $b = 0$ when compared to the classical setting --~ there is no dynamics on $\rho$, only a perturbation $v$ ~--. We investigate such a setting in order to adapt the Mortensen observer to this much more complex framework. The criterion we seek to minimize is then 

\begin{equation}
   \inf_{(\rho_0, v)} \mathcal{J}_t(\rho_0, v), 
    \label{INF}
\end{equation}

where 
\begin{equation} \mathcal{J}_t (\rho_0,v) = \mathcal{V}_0(\rho_0)+ \int_0^t \int_{\mathbb{R}^d} \frac{1}{2} |v(s,x)|^2 \,\mathrm{d}\rho_s \,\mathrm{d}s + \int_0^t \mathcal{E}(s, \rho_s) \,\mathrm{d}s,
\label{CRITERION}
\end{equation}

and with $\bm{\rho} = (\rho_s)_{0 \leq s \leq t}$ satisfying \eqref{SIMPLE} with initial value $\rho_0$. The first term corresponds to a correction related to the initial condition, the last term to a "fitness" with observations while the middle term is to be understood à la Benamou-Brenier i.e. as a mimic of the Wasserstein distance. We then consider the value function associated with \eqref{INF}, given by 

\begin{equation}
    \mathcal{V}(t,\rho) = \inf_{\substack{\bm{\rho} \in AC^2(0,t, \Pro) \\ \rho_t = \rho}} \left\{\mathcal{V}_0(\rho_0) + \int_0^t \int_{\mathbb{R}^d} \frac{1}{2} |v(s,x)|^2 \,\mathrm{d}\rho_s \,\mathrm{d}s + \int_0^t \mathcal{E}(s, \rho_s) \,\mathrm{d}s\right\}.
    \label{VALUE}
\end{equation}

The main result of this section is to prove that $\V$ is a viscosity solution to the HJB equation 

\begin{equation}
    \begin{dcases}
        \displaystyle \frac{\partial \V}{\partial t}(t, \mu) + \frac{1}{2} \|\nabla_\mu \V (t,\mu) \|_{\mathrm{L}^2(\mu)}^2 = \mathcal{E}(t,\mu), \\
       \V(0, \mu) = \mathcal{V}_0(\mu). 
    \end{dcases}
    \label{HJBW}
\end{equation}

A similar problem has been studied in the seminal work \cite{GANGBO} in a slightly different setting. Indeed, in \cite{GANGBO}, the authors managed to prove the fact that the value function is l.s.c. and a subsolution to the HJB equation for Lagrangians of the form $\mathcal{L}(t, \mu, v) = \frac{1}{2} \| v \|_{\mathrm{L}^2(\mu)}^2 - \mathcal{E}(\mu)$, and that it is indeed a viscosity solution for $d=1$. This result was completed in \cite{HYNDCOMP}, where they extended the results to an arbitrary dimension. In the following, we are studying similar albeit different Lagrangians, where the source term is positive --~which contrasts with the results obtained previously~-- and can depend on time. We are also able to show that there exist minima to $\V$, so that we can rigorously define 

\begin{equation}
    \hat{\mu}_t \in \argmin_{\nu \in \Pro} \V(t, \nu),
\end{equation}
which is the analog of the finite-dimensional Mortensen observer in the space of probability measures.

\section{Properties of the value function}

In this section, we wish to investigate the properties of the value function $\V$. We are interested, in the first place, in the regularity of the function, and, in a second place, by the dynamics it satisfies. We first show that it satisfies a dynamic programming principle, before tackling continuity results and existence of minimizers in the definition of $\V$. We finally use these three results in order to detail its link with Hamilton-Jacobi-Bellman equations on the space of probability measures.

\subsection{Bellman principle}

One of the key points associated with a HJB equation is the dynamic programming property, which we prove in the following lemma.

\begin{lemma}[Bellman principle]\label{DDP}
Assume that $\mathcal{V}_0$ is bounded from below. Then $\V$ satisfies, for $0 \leq s < t$,  

\begin{equation*}
    \V(t, \rho) = \inf_{\substack{\bm{\rho} \in AC^2(s,t, \Pro) \\ \rho_t = \rho}} \left\{  \V(s, \rho_s ) + \int_s^t \int_{\R^d}  \frac{1}{2} |v(\tau,x)|^2 \,\mathrm{d}\rho_\tau \,\mathrm{d}\tau + \int_s^t \mathcal{E}(\tau, \rho_\tau)  \,\mathrm{d}\tau \right\}.
\end{equation*}
\end{lemma}

\begin{proof}
Let $\rho \in \Pro$, $0 \leq s < t$ and $\varepsilon > 0$. There exists $\bm{\rho}^\varepsilon$ an $AC^2$ curve such that $\rho_t^\varepsilon = \rho$ and  

\begin{align*}
\mathcal{V}(t,\rho) + \varepsilon
&\ge \mathcal{V}_0(\rho_0^\varepsilon)
+ \int_0^t \int_{\mathbb{R}^d} \frac{1}{2} |v(\tau,x)|^2 \,\mathrm{d}\rho_\tau^\varepsilon \,\mathrm{d}\tau
+ \int_0^t \mathcal{E}(\tau, \rho_\tau^\varepsilon)\,\mathrm{d}\tau \notag \\
&= \mathcal{V}_0(\rho_0^\varepsilon)
+ \int_0^s \int_{\mathbb{R}^d} \frac{1}{2} |v(\tau,x)|^2 \,\mathrm{d}\rho_\tau^\varepsilon \,\mathrm{d}\tau
+ \int_0^s \mathcal{E}(\tau, \rho_\tau^\varepsilon)\,\mathrm{d}\tau \notag \\
&\quad
+ \int_s^t \int_{\mathbb{R}^d} \frac{1}{2} |v(\tau,x)|^2 \,\mathrm{d}\rho_\tau^\varepsilon \,\mathrm{d}\tau
+ \int_s^t \mathcal{E}(\tau, \rho_\tau^\varepsilon)\,\mathrm{d}\tau \notag \\
&\ge \mathcal{V}(s,\rho_s^\varepsilon)
+ \int_s^t \int_{\mathbb{R}^d} \frac{1}{2} |v(\tau,x)|^2 \,\mathrm{d}\rho_\tau^\varepsilon \,\mathrm{d}\tau
+ \int_s^t \mathcal{E}(\tau, \rho_\tau^\varepsilon)\,\mathrm{d}\tau \notag \\
&\ge \inf_{\substack{\bm{\rho} \in AC^2(s,t;\mathcal{P}) \\ \rho_t = \rho}}
\left\{
\mathcal{V}(s,\rho_s)
+ \int_s^t \int_{\mathbb{R}^d} \frac{1}{2} |v(\tau,x)|^2 \,\mathrm{d}\rho_\tau \,\mathrm{d}\tau
+ \int_s^t \mathcal{E}(\tau, \rho_\tau)\,\mathrm{d}\tau
\right\}.
\end{align*}

For the reverse inequality, we take $\bm{\rho} \in AC^2(s,t,\Pro)$ with $\rho_t = \rho$. Let $\bm{\rho}^\varepsilon \in AC^2(0,s, \Pro)$ such that $\rho_s^\varepsilon = \rho_s$ and 

\begin{equation*}
\V(s, \rho_s) + \varepsilon \geq  \mathcal{V}_0(\rho_0^\varepsilon) +\int_0^s \int_{\R^d}  \frac{1}{2} |v(\tau,x)|^2 \,\mathrm{d}\rho_\tau^\varepsilon \,\mathrm{d}\tau + \int_0^s \mathcal{E}(\tau, \rho_\tau^\varepsilon) \,\mathrm{d}\tau.
\end{equation*}

We extend $\bm{\rho}^\varepsilon \in AC^2([0,s], \Pro)$ to $[s,t]$ by setting $\rho^\varepsilon (\tau) = \rho(\tau)$ for $\tau \in [s,t]$ so that 

\begin{align*}
\V(t, \rho) &\leq \mathcal{V}_0(\rho_0^\varepsilon) +\int_0^t \int_{\R^d}  \frac{1}{2} |v(\tau,x)|^2 \,\mathrm{d}\rho_\tau^\varepsilon \,\mathrm{d}\tau + \int_0^t \mathcal{E}(\tau, \rho_\tau^\varepsilon) \,\mathrm{d}\tau \\
&= \mathcal{V}_0(\rho_0^\varepsilon) +\int_0^s \int_{\R^d}  \frac{1}{2} |v(\tau,x)|^2 \,\mathrm{d}\rho_\tau^\varepsilon \,\mathrm{d}\tau + \int_0^s \mathcal{E}(\tau, \rho_\tau^\varepsilon) \,\mathrm{d}\tau \\
&+ \int_s^t \int_{\R^d}  \frac{1}{2} |v(\tau,x)|^2 \,\mathrm{d}\rho_\tau^\varepsilon \,\mathrm{d}\tau + \int_s^t \mathcal{E}(\tau, \rho_\tau^\varepsilon) \,\mathrm{d}\tau \\
&\leq \V(s, \rho_s) + \int_s^t \int_{\R^d}  \frac{1}{2} |v(\tau,x)|^2 \,\mathrm{d}\rho_\tau \,\mathrm{d}\tau + \int_s^t \mathcal{E}(\tau, \rho_\tau) \,\mathrm{d}\tau + \varepsilon, 
\end{align*}

which ends the proof.
\end{proof}

\subsection{Continuity of the value function}

We then turn to continuity properties of the value function. We are able to show that $\V$ is continuous both in time and with respect to the Wasserstein topology. Continuity is indeed the kind of regularity we expect for the value function, when compared with the usual notion of viscosity solution in Hilbert spaces, see \cite{CRANDALL, CRANDALLLIONS} for instance.
\begin{lemma}
The function $(\mu, \nu) \longrightarrow \W_2^2(\mu, \nu)$ is lower semi-continuous (l.s.c.) on $\Pro \times \Pro$
\label{WLSC}
\end{lemma}
\begin{proof}
Let $\mu_m \rightharpoonup \mu$, $\nu_m \rightharpoonup \nu$, and $\gamma_m$ an optimal coupling between $\mu_m$ and $\nu_m$. As $\{\mu_m , m \in \mathbb{N}\}$ and $\{\nu_m , m \in \mathbb{N}\}$ are tight --~because they are weakly converging sequences~--, then, by \cite[Lemma 4.3]{VILLANI}, $\{\gamma_m , m \in \mathbb{N}\}$ is also tight. We then extract twice, so that 

\begin{equation*} 
\lim_{l \rightarrow +\infty} \W_2^2(\mu_{m_l}, \nu_{m_l}) = \liminf_{m \rightarrow + \infty} \, \W_2^2(\mu_m, \nu_m), 
\end{equation*}

\begin{equation*} 
\gamma_{m_l} \rightharpoonup \gamma \in \Pi(\mu, \nu),
\end{equation*}

since image measures through continuous functions pass to the limit. It means that  

\begin{align*}
\W_2^2(\mu, \nu) &\leq \int_{\R^d \times \R^d} |x-y|^2 \gamma(\,\mathrm{d}x,\,\mathrm{d}y) \notag \\
& \leq \liminf_{l \rightarrow + \infty} \int_{\R^d \times \R^d} |x-y|^2 \gamma_{m_l}(\,\mathrm{d}x,\,\mathrm{d}y) \notag \\
& = \lim_{l \rightarrow +\infty} \W_2^2(\mu_{m_l}, \nu_{m_l}) = \liminf_{m \rightarrow + \infty} \, \W_2^2(\mu_m, \nu_m),
\end{align*}

since $\mu \longrightarrow \int_{\R^d} f \,\mathrm{d}\mu$ is l.s.c. if $f$ is l.s.c. and bounded below, which gives the result thanks to \cite[Proposition 7.1.]{OTAM}.
\end{proof}
\begin{lemma}\label{CONTINUOUS}

Assume that:
\begin{enumerate}
\item $t \rightarrow \mathcal{E}(t, \mu)$ is continuous for every $\mu \in \Pro$,

\item $\mu \rightarrow \mathcal{E}(t, \mu)$ is weakly lower semi-continuous for every $t \in \R_+$,

\item $\mathcal{E}(t, \mu) \leq f(t) \psi\left( \int_{\R^d} |x|^2 \,\mathrm{d}\mu \right) + g(t)$ where $f, g, \psi$ are $\mathrm{L}_{\text{loc}}^\infty(\R_+, \R_+)$,

\item $\mathcal{V}_0$ is l.s.c. and non-negative.
\end{enumerate}

Then $\V$ is continuous on $\R_+ \times \Pro$ for the Wasserstein topology.
\end{lemma}

\begin{proof}
The main ideas for this proof are presented in \cite{GANGBO} and \cite{HYNDCOMP}. However, our source term now depends on time, which yields some extra technical difficulties. We are also able to prove lower-semi continuity due to the different form of the Lagrangian.
\begin{itemize}
    \item \begin{bf} $\V$ is upper semi-continuous. \end{bf}
Let $(\mu_n, t_n) \longrightarrow (\mu, t)$ and $\varepsilon > 0$. There exists $\bm{\rho} \in AC^2(0,t,\Pro)$ such that $\rho_t = \mu$ and such that 

\begin{equation}
\mathcal{V}_0(\rho_0)+ \int_0^t \int_{\mathbb{R}^d} \frac{1}{2} |v(s,x)|^2 \,\mathrm{d}\rho_s \,\mathrm{d}s + \int_0^t \mathcal{E}(s, \rho_s) \,\mathrm{d}s \leq \V(t, \mu) + \varepsilon.
\label{MINSEQ}
\end{equation}

Let $\bar{\bm{\rho}}^n \in AC^2(0,1,\Pro)$ be such that $\bar{\rho}_s^n$ is a geodesic connecting $\mu$ and $\mu_n$ with minimal velocity. In particular $\bar{\rho}_0^n = \mu$, $\bar{\rho}_1^n = \mu_n$. We define 

\begin{equation*}
\rho_s^n = 
\begin{dcases}
\rho_{\frac{n t s}{(n-1)t_n}} &\text{ for } s\in \left[0, \frac{(n-1) t_n}{n} \right], \\
\bar{\rho}_{\frac{ns}{t_n} + (1-n)}^n &\text{ for } s\in \left[\frac{(n-1) t_n}{n}, t_n \right].
\end{dcases}
\end{equation*}

In particular, $\rho_{t_n}^n = \mu_n$. Then, with $v^n$ the velocity associated to $\bm{\rho}^n$, 

\begin{multline*}
\int_0^{t_n} \int_{\mathbb{R}^d} \frac{1}{2} |v^n(s,x)|^2 \,\mathrm{d}\rho_s^n \,\mathrm{d}s
+ \int_0^{t_n} \mathcal{E}(s, \rho_s^n)\,\mathrm{d}s 
= \int_0^{\frac{(n-1)t_n}{n}} \int_{\mathbb{R}^d} \frac{1}{2} |v^n|^2 \,\mathrm{d}\rho_s^n \,\mathrm{d}s
+ \int_0^{\frac{(n-1)t_n}{n}} \mathcal{E}(s, \rho_s^n)\,\mathrm{d}s \notag \\
\quad + \int_{\frac{(n-1)t_n}{n}}^{t_n} \int_{\mathbb{R}^d} \frac{1}{2} |v^n|^2 \,\mathrm{d}\rho_s^n \,\mathrm{d}s
+ \int_{\frac{(n-1)t_n}{n}}^{t_n} \mathcal{E}(s, \rho_s^n)\,\mathrm{d}s.
\end{multline*}

We first have that

\begin{align*}
\int_0^{\frac{(n-1)t_n}{n}} &\int_{\mathbb{R}^d} \frac{1}{2} |v^n|^2 \,\mathrm{d}\rho_s^n \,\mathrm{d}s
+ \int_0^{\frac{(n-1)t_n}{n}} \mathcal{E}(s, \rho_s^n)\,\mathrm{d}s \notag \\
&= \frac{nt}{(n-1)t_n}
\int_0^{t} \int_{\mathbb{R}^d} \frac{1}{2} |v|^2 \,\mathrm{d}\rho_s \,\mathrm{d}s
+ \frac{(n-1)t_n}{nt}\int_0^{t} \mathcal{E}\!\left(\frac{(n-1)t_n}{nt}s, \rho_s\right)\,\mathrm{d}s
,
\end{align*}

but also that

\begin{align*}
\int_{\frac{(n-1)t_n}{n}}^{t_n} &\int_{\mathbb{R}^d} \frac{1}{2} |v^n|^2 \,\mathrm{d}\rho_s^n \,\mathrm{d}s
+ \int_{\frac{(n-1)t_n}{n}}^{t_n} \mathcal{E}(s, \rho_s^n)\,\mathrm{d}s \notag \\
&= \frac{t_n}{n}
\Biggl[
\int_0^{1} \int_{\mathbb{R}^d} \frac{1}{2} |\bar{v}^n|^2 \,\mathrm{d}\bar{\rho}_s^n \,\mathrm{d}s
+ \int_0^{1} \mathcal{E}\!\left(\frac{(s+n-1)t_n}{n}, \bar{\rho}_s^n\right)\,\mathrm{d}s
\Biggr] \notag \\
&= \frac{t_n}{2n}\mathcal{W}_2^2(\mu,\mu_n)
+ \frac{t_n}{n} \int_0^{1} \mathcal{E}\!\left(\frac{(s+n-1)t_n}{n}, \bar{\rho}_s^n\right)\,\mathrm{d}s.
\end{align*}

We define $I_1$ and $I_2$ as 

\begin{equation*}
    I_1 = \frac{nt}{(n-1)t_n}
\int_0^{t} \int_{\mathbb{R}^d} \frac{1}{2} |v|^2 \,\mathrm{d}\rho_s \,\mathrm{d}s
+ \frac{(n-1)t_n}{nt}\int_0^{t} \mathcal{E}\!\left(\frac{(n-1)t_n}{nt}s, \rho_s\right)\,\mathrm{d}s
,
\end{equation*}
\begin{equation*}
    I_2 = \frac{t_n}{2n} \mathcal{W}_2^2(\mu, \mu_n) 
+ \frac{t_n}{n} \int_0^{1} \mathcal{E}\left(\frac{(s+n-1)t_n}{n}, \bar{\rho}_s^n\right) \,\mathrm{d}s.
\end{equation*}
But $\W_2(\bar{\rho}_s^n, \mu) \leq \W_2(\mu_n, \mu) \leq C$, so that $\{ \bar{\rho}_s^n \, , \, s \in [0,1] \}$ is bounded uniformly in $n$ and in $s$ which yields, as $\psi \in L^\infty(\R_+)$

\begin{equation*}
    \psi\left( \int_{\mathbb{R}^d} |x|^2 \, \mathrm{d}\bar{\rho}^n_s\right) \leq C'.
\end{equation*}

Then 

\begin{multline*}
\int_0^{1} \mathcal{E}\left(\frac{(s+n-1)t_n}{n}, \bar{\rho}_s^n\right) \,\mathrm{d}s \leq \int_0^{1} f\left( \frac{(s + n -1)t_n}{n}\right) \psi\left( \int_{\mathbb{R}^d} |x|^2 \, \mathrm{d}\bar{\rho}^n_s \right) + g\left( \frac{(s + n -1)t_n}{n}\right)\,\mathrm{d}s \\
\leq \int_0^1 C'f\left( \frac{(s + n -1)t_n}{n}\right) +  g\left( \frac{(s + n -1)t_n}{n}\right) \,\mathrm{d}s,
\end{multline*}

so that 

\begin{align*}
I_2 &\leq \frac{C t_n}{n}+ \frac{t_n}{n} \int_{0}^1 C'f\left( \frac{(s + n -1)t_n}{n}\right) +  g\left( \frac{(s + n -1)t_n}{n}\right) \,\mathrm{d}s \notag\\
& \leq \frac{(C + C') \| f\|_{\infty, loc} + \| g\|_{\infty, loc}}{n} \longrightarrow 0,
\end{align*}
 For $I_1$, we begin by observing that

\begin{equation*}
\mathcal{E}\left(\frac{(n-1)t_n s}{nt}, \rho_s\right) \leq C' f\left(\frac{(n-1)t_n s}{n t} \right) + g\left(\frac{(n-1)t_n s}{n t}\right) \leq C'',
\end{equation*}

Then, by reverse Fatou's lemma and upper semi-continuity of $\mathcal{E}(\cdot, \rho_s)$, 

\begin{align*}
\limsup_{n \rightarrow + \infty} I_1 &\leq \int_0^{t} \int_{\R^d} \frac{1}{2} |v(s,x)|^2 \,\mathrm{d}\rho_s \,\mathrm{d}s + \int_0^{t} \limsup_{n \rightarrow + \infty}  \mathcal{E}\left(\frac{(n-1)t_n s}{nt}, \rho_s\right) \,\mathrm{d}s \notag \\
&\leq \int_0^{t} \int_{\mathbb{R}^d} \frac{1}{2} |v(s,x)|^2 \,\mathrm{d}\rho_s \,\mathrm{d}s + \int_0^{t}  \mathcal{E}\left(s, \rho_s\right) \,\mathrm{d}s,
\end{align*}

so that, in the end 
\begin{equation*}
\limsup_{n \rightarrow + \infty} \int_0^{t_n} \int_{\mathbb{R}^d} \frac{1}{2} |v^n(s,x)|^2 \,\mathrm{d}\rho_s^n \,\mathrm{d}s + \int_0^{t_n} \mathcal{E}\left(s, \rho_s^n\right)\,\mathrm{d}s 
\leq \int_0^{t} \int_{\mathbb{R}^d} \frac{1}{2} |v(s,x)|^2 \,\mathrm{d}\rho_s \,\mathrm{d}s +  \int_0^{t}  \mathcal{E}\left(s, \rho_s\right) \,\mathrm{d}s.
\end{equation*}

Since $\rho_0^n = \rho_0$, it yields 

\begin{align*}
\limsup_{n \rightarrow + \infty} \V(t_n, \mu_n) &\leq \limsup_{n \rightarrow + \infty} \left[ \mathcal{V}_0(\rho_0^n) + \int_0^{t_n} \int_{\mathbb{R}^d} \frac{1}{2} |v^n(s,x)|^2 \,\mathrm{d}\rho_s^n \,\mathrm{d}s + \int_0^{t_n} \mathcal{E}(s, \rho_s^n) \,\mathrm{d}s \right] \notag \\
 &\leq \mathcal{V}_0(\rho_0) + \int_0^{t} \int_{\mathbb{R}^d} \frac{1}{2} |v(s,x)|^2 \,\mathrm{d}\rho_s \,\mathrm{d}s +  \int_0^{t}  \mathcal{E}\left(s, \rho_s\right)\,\mathrm{d}s. 
\end{align*}

It follows from \eqref{MINSEQ} that 

\begin{equation*}
\limsup_{n \rightarrow + \infty} \V(t_n, \mu_n) \leq \V(t,\mu) + \varepsilon,
\end{equation*}

which proves the upper semi-continuity of $\V$.
\smallbreak
\item \begin{bf} $\V$ is lower semi-continuous. \end{bf}Let $(t_n, \mu_n) \rightarrow (t, \mu)$ and $\varepsilon_n \rightarrow 0$. We have $\bm{\rho}^n \in AC^2(0,t_n)$ with $\rho_{t_n} = \mu_n$ such that

\begin{equation*}
    \V(t_n, \mu_n) + \varepsilon_n \geq \mathcal{V}_0(\rho_0^n) + \frac{1}{2} \int_0^{t_n} \int_{\R^d} |v^n(s,x)|^2 \,\mathrm{d}\rho_s^n \, \mathrm{d}s + \int_0^{t_n} \mathcal{E}(s, \rho_s^n) \mathrm{d}s.
\end{equation*}

Let $\tilde{\bm{\rho}}^n \in AC^2(0,1, \Pro)$ be a geodesic connecting $\mu_n$ to $\mu$. We define $\bar{\bm{\rho}}^n$ with

\begin{equation*}
    \bar{\rho}_s^n =
    \begin{dcases}
    \rho^n_{\frac{n t_n}{(n-1)t}s} & \text{ for } s \in [0, (1 -\frac{1}{n})t],\\
    \tilde{\rho}^n_{\frac{ns}{t} + (1-n)} & \text{ for } s \in [(1 -\frac{1}{n})t, t].
    \end{dcases}
\end{equation*}

In particular,

\begin{align*}
    \V(t, \mu) & \leq \mathcal{V}_0(\bar{\rho}_0^n) + \frac{1}{2} \int_0^{t} \int_{\R^d} |\bar{v}^n(s,x)|^2 \,\mathrm{d}\bar{\rho}_s^n \, \mathrm{d}s + \int_0^{t} \mathcal{E}(s, \bar{\rho}_s^n) \mathrm{d}s \notag \\
    & = \mathcal{V}_0(\rho_0^n) + \frac{(n-1)t}{2nt_n} \int_0^{t_n} \int_{\R^d} |v^n(s,x)|^2 \,\mathrm{d}\rho_s^n \, \mathrm{d}s + \frac{(n-1)t}{nt_n}\int_0^{t_n} \mathcal{E}\left(\frac{(n-1)t s}{nt_n}, \rho_s^n\right) \mathrm{d}s \notag \\
    & + \frac{t}{2n} \int_0^{1} \int_{\R^d} |v^n(s,x)|^2 \,\mathrm{d}\tilde{\rho}_s^n \, \mathrm{d}s + \frac{t}{n}\int_0^{1} \mathcal{E}\left(\frac{(s+ n -1)t}{n}, \tilde{\rho}_s^n\right) \mathrm{d}s.
\end{align*}

Let $\delta >0$. For $n$ large enough, we get

\begin{align*}
    \V(t, \mu) & \leq \mathcal{V}_0(\rho_0^n) + \frac{1+\delta}{2} \int_0^{t_n} \int_{\R^d} |v^n(s,x)|^2 \,\mathrm{d}\rho_s^n \, \mathrm{d}s + (1 + \delta)\int_0^{t_n} \mathcal{E}\left(s, \rho_s^n\right) \mathrm{d}s \notag \\
    & +\frac{t}{2n} \W_2^2(\mu, \mu_n) + \frac{t}{n} \int_0^{1} \mathcal{E}\left(\frac{(s+ n -1)t}{n}, \tilde{\rho}_s^n\right) \mathrm{d}s \notag \\
    & + (1+\delta) \left( \int_0^{t_n} \mathcal{E}\left(s, \rho_s^n\right) \mathrm{d}s - \int_0^{t_n} \mathcal{E}\left(\frac{(n-1)t s}{nt_n}, \rho_s^n\right) \mathrm{d}s \right) \notag \\
    & \leq (1+ \delta) \V(t_n, \mu_n) + (1+ \delta) \varepsilon_n + \frac{t}{2n} \W_2^2(\mu, \mu_n) + \frac{t}{n} \int_0^{1} \mathcal{E}\left(\frac{(s+ n -1)t}{n}, \tilde{\rho}_s^n\right) \mathrm{d}s \notag \\
    & + (1+ \delta) \left( \int_0^{t_n} \mathcal{E}\left(s, \rho_s^n\right) \mathrm{d}s - \int_0^{t_n} \mathcal{E}\left(\frac{(n-1)t s}{nt_n}, \rho_s^n\right) \mathrm{d}s \right).
\end{align*}

because $\frac{(n-1)t}{n t_n} \rightarrow 1$. But 

\begin{equation*}
    \int_0^{t_n} \mathcal{E}\left(s, \rho_s^n\right) \mathrm{d}s - \int_0^{t_n} \mathcal{E}\left(\frac{(n-1)t s}{nt_n}, \rho_s^n\right) \mathrm{d}s \rightarrow 0,
\end{equation*}

since $\mathcal{E}$ is continuous hence uniformly continuous on $[0, \max(t_n,t)]$. We have shown in the previous part that $\frac{t}{2n} \W_2^2(\mu, \mu_n) + \frac{t}{n} \int_0^{1} \mathcal{E}\left(\frac{(s+ n -1)t}{n}, \tilde{\rho}_s^n\right) \mathrm{d}s$ goes to 0. Letting $n \rightarrow + \infty$, we get

\begin{equation*}
    \V(t, \mu) \leq (1+ \delta) \liminf_{n \rightarrow + \infty}  \V(t_n, \mu_n)
\end{equation*}

for every $\delta >0$, so that the result follows.
\end{itemize}
\end{proof}

\subsection{Existence of minimizing curves and of minimizers}

In this section, we prove that the infimum in the definition of $\V$ is in fact a minimum, so that there exists a minimizing curve $\bm{\rho} \in AC^2(0,t,\Pro)$ such that $\rho_t = \mu$ and such that $\V(t, \mu) = \mathcal{J}_t(\bm{\rho})$. This kind of result is well known in the literature for classical HJB equations, for instance in \cite{FATHI, EVANS}. We will use this result in order to derive the HJB equation in the next section.

\begin{lemma}\label{EXIST}
Assume that: 

\begin{enumerate}
\item $t \rightarrow \mathcal{E}(t, \mu)$ is continuous for every $\mu \in \Pro$

\item $\mu \rightarrow \mathcal{E}(t, \mu)$ is weakly l.s.c. for every $t \in \R_+$

\item $\mathcal{V}_0$ is l.s.c. and non-negative.

\end{enumerate}
Then the infimum in the definition of $\V$ is a minimum.

\end{lemma}

\begin{proof}

Since $\V_0$ is proper, we know that $\V$ is not identically equal to $+\infty$. Let $(\bm{\rho}^n)_{n \in \mathbb{N}}$ be a minimizing sequence in the definition of $\V$. We have in particular that, for $n$ large enough 

\begin{equation*}
\mathcal{V}_0(\rho_0^n) + \int_0^t \int_{\mathbb{R}^d} \frac{1}{2} |v^n(s,x)|^2 \,\mathrm{d}\rho_s^n \,\mathrm{d}s + \int_0^t \mathcal{E}(s, \rho_s^n) \,\mathrm{d}s \leq \V(t, \mu) + \varepsilon,
\end{equation*}

which gives that 

\begin{equation*}
\sup_{n \in \mathbb{N}} \int_0^t \|v_s^n \|_{\mathrm{L}^2(\rho_s^n)}^2 \,\mathrm{d}s < + \infty \text{  and  } \sup_{n \in \mathbb{N}} \W_2^2(\rho_0^n, \mu) < + \infty,
\end{equation*}

since 

\begin{equation*}
\W_2^2(\rho_0^n, \mu) \leq \frac{t}{2} \int_0^t \|v_s^n \|_{\mathrm{L}^2(\rho_s^n)}^2 \,\mathrm{d}s.
\end{equation*}
We use \cite[Proposition 4]{GANGBOPOISSON} to get the existence of $\bm{\rho} \in AC^2(0,t, \Pro)$ such that, up to a subsequence,

\begin{equation*}
\rho_s^n \rightharpoonup \rho_s \; \text{ for every } s\in[0,t].
\end{equation*}

In particular, $\rho_t = \mu$ by uniqueness of the weak limit. Then, by lower-semi-continuity and weak convergence 

\begin{equation*}
\liminf_{n \rightarrow + \infty} \, \mathcal{V}_0(\rho_0^n) \geq \mathcal{V}_0(\rho_0),
\end{equation*}

\begin{equation*}
\liminf_{n \rightarrow + \infty} \int_0^t \mathcal{E}(s, \rho_s^n) \,\mathrm{d}s \geq \int_0^t \mathcal{E}(s, \rho_s) \,\mathrm{d}s.
\end{equation*}

Finally, \cite[Proposition 3]{GANGBOPOISSON} together with the bound on $\W_2(\rho_0^n, \delta_0)$ gives 

\begin{equation*}
\liminf_{n \rightarrow + \infty} \int_0^t \|v_s^n \|_{\rho_s^n}^2 \,\mathrm{d}s \geq \int_0^t \|v_s \|_{\rho_s}^2 \,\mathrm{d}s,
\end{equation*}

which means that 

\begin{align*}
\liminf_{n \rightarrow + \infty} \; & \mathcal{V}_0(\rho_0^n) + \int_0^t \int_{\mathbb{R}^d} \frac{1}{2} |v(s,x)|^2 \,\mathrm{d}\rho_s^n \,\mathrm{d}s + \int_0^t \mathcal{E}(s, \rho_s^n) \,\mathrm{d}s \notag \\
& \geq \mathcal{V}_0(\rho_0) + \int_0^t \int_{\mathbb{R}^d} \frac{1}{2} |v(s,x)|^2 \,\mathrm{d}\rho_s \,\mathrm{d}s + \int_0^t \mathcal{E}(s, \rho_s) \,\mathrm{d}s,
\end{align*}

which shows that $\rho_s$ is a minimizer.
\end{proof}

\begin{remark}
    One has to understand \cite[Proposition 4]{GANGBOPOISSON} as some Arzelà-Ascoli theorem for $AC^2$ curves in $\Pro$ which gives the existence of a limiting curve and \cite[Proposition 3]{GANGBOPOISSON} as a lower semi-continuity result on the velocity of such curves. This is reminiscent of the tools used in a more classical setting in order to prove existence of minimizers.
\end{remark}

\begin{theorem} \label{MINI}
    Assume that: 

\begin{enumerate}
\item $t \rightarrow \mathcal{E}(t, \mu)$ is continuous for every $\mu \in \Pro$

\item $\mu \rightarrow \mathcal{E}(t, \mu)$ is weakly l.s.c. for every $t \in \R_+$

\item $\mathcal{V}_0$ is l.s.c. and satisfies, for some $\nu \in \Pro$

\begin{equation*}
    \V_0(\mu) \geq L \W_2(\mu, \nu)^2.
\end{equation*}

\end{enumerate}

    Then, for every $t \geq 0$, there exists a minimum $\hat{\mu}_t$ to $\V(t, \cdot)$ on $\Pro$. We define this minimum as the Mortensen observer on the space of probability measures.
\end{theorem}

\begin{proof}
    As $\V(t, \mu) \geq 0$ for every $t\geq 0$ and $\mu \in \Pro$, let us take $(\mu_n)_{n \in \mathbb{N}}$ a minimizing sequence for $\V(t, \cdot)$. Let $\bm{\rho}^n \in AC^2(0,t,\Pro)$ be a minimizer in the definition of $\V(t, \mu_n)$. As $(\V(t, \mu_n))_{n\in \mathbb{N}}$ converges, we get that

    \begin{equation*}
       \int_0^t \int_{\R^d} |v^n(s,x)|^2 \dr \rho_s^n \dr s\leq \V(t, \mu_n) \leq C,
    \end{equation*}

    and

    \begin{equation*}
        \W_2^2(\rho_0^n, \nu) \leq C'.
    \end{equation*}

    Thanks to \cite[Proposition 4]{GANGBOPOISSON}, we get $\bm{\rho} \in AC^2(0,t,\Pro)$ such that, up to a subsequence, $\bm{\rho}^n$ converges towards $\bm{\rho}$ and 
    
    \begin{equation*}
        \int_{\R^d} |x|^2 \dr \mu_n \leq 3\W_2^2(\mu_n, \rho_0^n) +  3\W_2^2(\nu, \rho_0^n)+ 3\int_{\R^d} |x|^2 \dr \nu \leq C'',
    \end{equation*}

    so that $\mu_n$ converges towards $\mu$ not only weakly, but also with respect to the Wasserstein topology, also up to a subsequence. Since $\V$ is continuous with respect to the Wasserstein topology, we get the result we were looking for.
\end{proof}

\subsection{A first geometric approach to viscosity solutions in Wasserstein spaces}

In this section, we study the viscosity properties of the value function following the approach of \cite{GANGBO}. This method, which leverages the geometric structure of the space of probability measures, represents the original framework in which viscosity solutions were introduced. However, this notion does not satisfy a comparison principle — the fundamental feature of classical viscosity solutions. Despite this limitation, these classical properties remain essential for establishing more suitable viscosity properties of the value function, which is why we present them in detail below.

\begin{definition}[\cite{GANGBO}]
Let $\mu \in \Pro$, $\xi \in \overline{\left\{ \nabla \phi, \, \phi \in C_c^\infty(\R^d)  \right\}}^{\mathrm{L}^2(\mu)}$, and $\mathcal{F}  \, : \, \Pro \longrightarrow \R$. 
\begin{enumerate}
\item We say that $\xi$ belongs to the classical subdifferential of $\mathcal{F}$ and we write $\xi \in \partial_. \mathcal{F}(\mu)$ if 

\begin{equation*}
\mathcal{F}(\nu) - \mathcal{F}(\mu) \geq \sup_{\gamma \in \Pi_o(\mu, \nu)} \int_{\R^d \times \R^d} \langle \xi(x), (y-x)\rangle \gamma(\,\mathrm{d}x , \,\mathrm{d}y) + o\left(\W_2(\mu, \nu)\right) \quad \forall \nu \in \Pro.
\end{equation*}

\item We say that $\xi$ belongs to the classical superdifferential of $\mathcal{F}$ and we write $\xi \in \partial^. \mathcal{F}(\mu)$ if $-\xi \in \partial_. (- \mathcal{F})(\mu)$

\end{enumerate}
\end{definition}

\begin{definition}[\cite{GANGBO}] \label{GEOMVIS}
Let $\mathcal{V} :  \R_+ \times \Pro \longrightarrow \R$, $\mu \in \Pro$ and $t \geq 0$.

\begin{enumerate}
    \item We say that $\mathcal{V}$ is a classical viscosity subsolution to \eqref{HJBW} if $\mathcal{V}$ is u.s.c., $\mathcal{V}(0, \mu) \leq \mathcal{V}_0(\mu) $ and
\begin{equation*}
\theta + \frac{1}{2} \| \xi \|_{\mathrm{L}^2(\mu)}^2 - \mathcal{E}(t,\mu) \leq 0 \quad \forall (t, \mu) \in \R_+ \times \Pro, \; (\theta, \xi) \in \partial^. \mathcal{V}(t, \mu) .
\end{equation*}

\item We say that $\mathcal{V}$ is a classical viscosity supersolution to \eqref{HJBW} if $\mathcal{V}$ is l.s.c., $\mathcal{V}(0, \mu) \geq \mathcal{V}_0(\mu) $ and
\begin{equation*}
\theta + \frac{1}{2} \| \xi \|_{\mathrm{L}^2(\mu)}^2 -  \mathcal{E}(t,\mu) \geq 0 \quad \forall (t, \mu) \in \R_+ \times \Pro, \; (\theta, \xi) \in \partial_. \mathcal{V}(t, \mu) .
\end{equation*}
\end{enumerate}
We say that $\V$ is a viscosity solution to \eqref{HJBW} if $\V$ is both a classical viscosity subsolution and a classical viscosity supersolution.
\end{definition}

We are now able to prove the following result.

\begin{theorem}\label{VISCOSOLW}
Assume the following: 

\begin{enumerate}
\item $t \rightarrow \mathcal{E}(t, \mu)$ is continuous for every $\mu \in \Pro$,

\item $\mu \rightarrow \mathcal{E}(t, \mu)$ is weakly lower semi-continuous for every $t \in \R_+$,

\item $t \rightarrow \mathcal{E}(t, (\id + (T - t)\nabla \phi)_\# \mu) $ is upper semi-continuous for every $\phi \in C_c^\infty(\R^d)$ and $T>t$ (or at least for $t$ close enough to $T)$,

\item $\mathcal{E}(t, \mu) \leq f(t) \psi\left( \int_{\R^d} |x|^2 \,\mathrm{d}\mu \right) + g(t)$ where $f, g, \psi$ are in $L_{loc}^\infty(\R_+)$,

\item $\mathcal{V}_0$ is l.s.c. and non-negative.

\end{enumerate}  

Then $\V$ is a classical viscosity solution to 

\begin{equation*}
    \begin{dcases}
        \partial_t \V(t, \mu) + \frac{1}{2} \left\|\nabla_\mu \V (t,\mu) \right\|_{\mathrm{L}^2(\mu)}^2 = \mathcal{E}(t, \mu), \\
        \V(0, \mu) = \mathcal{V}_0(\mu) .
    \end{dcases}
\end{equation*}
\end{theorem}

\begin{proof}

Under the assumptions of our theorem, we already know with \autoref{CONTINUOUS} that $\V$ is continuous on $\R_+ \times \Pro$. We obviously have that $\V(0, \mu) = \mathcal{V}_0(\mu)$ for all $\mu \in \Pro$, so the first inequalities in \autoref{GEOMVIS} are proven. 

\begin{itemize}
    \item \textbf{$\V$ is a classical subsolution. } Let $(t,\mu) \in \R_+ \times \Pro$, and $(\theta, \xi) \in \partial^.\V(t, \mu)$. Let $\varphi \in C_c^\infty(\R^d)$, $\varepsilon > 0$ small enough so that $x \rightarrow \frac{|x|^2}{2} + \lambda \varphi(x)$ is strictly convex for all $\lambda \in [0, \varepsilon]$. We define 
    
    \begin{equation*}
    \phi_s^\varepsilon(x) = \frac{|x|^2}{2} + (t - s) \varphi(x)
    \end{equation*}

which is strictly convex for $s \in [t- \varepsilon, t]$. Let 

\begin{equation*}
T^\varepsilon = \text{id } + \varepsilon \nabla \varphi,  \quad \mu^\varepsilon = T^\varepsilon_\# \mu, \quad \rho_s^\varepsilon = (\nabla \phi_s^\varepsilon)_\# \mu \text{ for } s \in [t- \varepsilon, t].
\end{equation*}

We extend $\rho_s^\varepsilon$ to $[0, t-\varepsilon]$ by taking $\bm{\rho}^\epsilon \in AC^2(0,t-\varepsilon, \Pro)$ such that $\rho_{t - \varepsilon}^\varepsilon = \mu^\varepsilon$ and 

\begin{equation*}
\V(t-\varepsilon, \mu^\varepsilon) + \varepsilon^2 \geq \mathcal{V}_0(\rho^\varepsilon_0)+ \int_0^{t - \varepsilon} \int_{\mathbb{R}^d} \frac{1}{2} |v^\varepsilon(s,x)|^2 \,\mathrm{d}\rho^\varepsilon_s \,\mathrm{d}s + \int_0^{t - \varepsilon} \mathcal{E}(s, \rho_s^\varepsilon) \,\mathrm{d}s.
\end{equation*}

Then $\bm{\rho}^\varepsilon \in AC^2(0,t, \Pro)$ and $\rho_t^\varepsilon = \mu$, hence 

\begin{align}
\V(t, \mu) &\leq \mathcal{V}_0(\rho_0^\varepsilon)+ \int_0^t \int_{\mathbb{R}^d} \frac{1}{2} |v^\varepsilon(s,x)|^2 \,\mathrm{d}\rho_s^\varepsilon \,\mathrm{d}s + \int_0^t \mathcal{E}(s, \rho_s^\varepsilon) \,\mathrm{d}s \notag \\
&\leq \V(t-\varepsilon, \mu^\varepsilon) + \int_{t- \varepsilon}^t \int_{\mathbb{R}^d} \frac{1}{2} |v^\varepsilon(s,x)|^2 \,\mathrm{d}\rho_s^\varepsilon \,\mathrm{d}s + \int_{t - \varepsilon}^t \mathcal{E}(s, \rho_s^\varepsilon) \,\mathrm{d}s + \varepsilon^2.
\label{EPSBOUND}
\end{align}

Here, $v^\varepsilon$ is given by 

\begin{equation*}
v^\varepsilon(s, \cdot) = \frac{\nabla {\phi_s^\varepsilon}^* - \text{ id}}{t-s},
\end{equation*}

where ${\phi_s^\varepsilon}^*$ is the Legendre transform of $\phi_s^\varepsilon$. Then, for $s\in [t- \varepsilon, t]$, we have

\begin{equation*}
\W_2(\rho^\varepsilon_s, \mu) = \left( \int_{\R^d} |\nabla \phi_s^\varepsilon - x|^2 \,\mathrm{d}\rho_s^\varepsilon \right)^\frac{1}{2}= (t-s) \left( \int_{\R^d} |\nabla \varphi|^2 \,\mathrm{d}\mu \right)^\frac{1}{2},
\end{equation*}

which yields 

\begin{equation*}
\int_{t- \varepsilon}^t \int_{\mathbb{R}^d} \frac{1}{2} |v^\varepsilon(s,x)|^2 \,\mathrm{d}\rho_s^\varepsilon \,\mathrm{d}s = \frac{\W_2^2(\rho_{t - \varepsilon}^\varepsilon, \mu)}{2\varepsilon} = \frac{\varepsilon}{2} \int_{\R^d} |\nabla \varphi|^2 \,\mathrm{d}\mu.
\end{equation*}

Similarly, we have

\begin{equation*}
 \int_{t - \varepsilon}^t \mathcal{E}(s, \rho_s^\varepsilon) \,\mathrm{d}s = \int_{t - \varepsilon}^t \mathcal{E}(s, (\text{id }+ (t-s) \nabla \varphi )_\#\mu) \,\mathrm{d}s.
\end{equation*}

Therefore, since $s \rightarrow \mathcal{E}(s, (\text{id }+ (t-s) \nabla \varphi )_\#\mu)$ is  upper semi-continuous thanks to our assumptions,
\begin{equation*}
\limsup_{\varepsilon \rightarrow 0} \frac{1}{\varepsilon} \int_{t - \varepsilon}^t \mathcal{E}(s,(\text{id }+ (t-s) \nabla \varphi )_\#\mu) \,\mathrm{d}s \leq \mathcal{E}(t, \mu).
\end{equation*}

In the end, together with \eqref{EPSBOUND}, 

\begin{equation*}
\limsup_{\varepsilon \rightarrow 0} \frac{\V(t, \mu) - \V(t-\varepsilon, \mu^\varepsilon)}{\varepsilon} \leq \frac{1}{2} \int_{\R^d} |\nabla \varphi|^2 \,\mathrm{d}\mu + \mathcal{E}(t, \mu).
\end{equation*}

Since $(\theta, \xi) \in \partial^.\V (t, \mu)$,

\begin{multline*}
\V(t- \varepsilon, \mu^\varepsilon) - \V(t, \mu) \leq \varepsilon \int_{\R^d \times \R^d} \langle  \xi(x), (y-x)\rangle \gamma^\varepsilon(\,\mathrm{d}x , \,\mathrm{d}y) - \varepsilon \theta + o\left(\W_2(\mu, \mu^\varepsilon)\right) + o(\varepsilon) \\= \varepsilon \int_{\R^d} \langle \xi, \nabla \varphi \rangle \,\mathrm{d}\mu- \varepsilon \theta + o(\varepsilon),
\end{multline*}

where $\gamma^\varepsilon = (\text{id } \times T^\varepsilon)_\# \mu \in \Pi_o(\mu, \mu^\varepsilon)$ and since $\W_2(\mu, \mu^\varepsilon) = \varepsilon \left( \int_{\R^d} |\nabla \varphi|^2 \,\mathrm{d}\mu \right)^\frac{1}{2} $. Then we have that

\begin{equation*}
\liminf_{\varepsilon \rightarrow 0} \frac{\V(t, \mu) - \V(t-\varepsilon, \mu^\varepsilon)}{\varepsilon} \geq  \theta - \int_{\R^d} \langle \xi, \nabla \varphi \rangle \,\mathrm{d}\mu.
\end{equation*}
This yields

\begin{align*}
\theta + \int_{\R^d} \frac{1}{2} |\xi(x)|^2 \,\mathrm{d}\mu  -\mathcal{E}(t, \mu) & \leq \frac{1}{2} \int_{\R^d} |\nabla \varphi|^2 \,\mathrm{d}\mu + \int_{\R^d} \frac{1}{2} |\xi(x)|^2 \,\mathrm{d}\mu+  \int_{\R^d} \langle \xi, \nabla \varphi \rangle \,\mathrm{d}\mu \notag \\
& = \frac{1}{2} \int_{\R^d} |\xi + \nabla \varphi|^2 \,\mathrm{d}\mu.
\end{align*}

Since $\varphi \in C_c^\infty(\R^d)$ is arbitrary and $\{\nabla \varphi, \;   \varphi \in  D(\R^d)\}$ is dense in $T_\mu \Pro$ which contains $\xi$, we conclude that $\theta + \int_{\R^d} \frac{1}{2} |\xi(x)|^2 \,\mathrm{d}\mu  - \mathcal{E}(t, \mu) \leq 0$ which proves that $\V$ is a subsolution to    \eqref{HJBW}.

\item \textbf{$\V$ is a classical supersolution.} Let $(t,\mu) \in \R_+ \times \Pro$, and $(\theta, \xi) \in \partial_.\V(t, \mu)$. Let $\bm{\rho}$ be a minimizer for $\V$ --~which exists thanks to \autoref{EXIST}~--. Let $0< \varepsilon < t$, and $\gamma^\varepsilon \in \Pi_o (\rho_t, \rho_{t - \varepsilon})$. Thanks to Bellman principle, \autoref{DDP}, we have 

\begin{equation*}
\V(t, \mu) - \V(t-\varepsilon, \rho_{t-\varepsilon}) = \int_{t- \varepsilon}^t \int_{\R^d}  \frac{1}{2} |v(s,x)|^2 \,\mathrm{d}\rho_s \,\mathrm{d}s + \int_{t - \varepsilon}^t \mathcal{E}(s, \rho_s) \,\mathrm{d}s.
\end{equation*}

Since $(\theta, \xi) \in \partial_.\V(t, \mu)$, 

\begin{equation*}
\mathcal{V}(t- \varepsilon, \rho_{t - \varepsilon}) - \mathcal{V}(t,\mu) \geq \int_{\R^d \times \R^d} \langle \xi(x), (y-x)\rangle \gamma^\varepsilon(\,\mathrm{d}x , \,\mathrm{d}y) - \theta \varepsilon + o(\varepsilon) + o\left(\W_2(\mu, \rho_{t - \varepsilon})\right).
\end{equation*}

It yields

\begin{multline*}
   0 \geq  \int_{t- \varepsilon}^t \int_{\R^d}  \frac{1}{2} |v(s,x)|^2 \,\mathrm{d}\rho_s \,\mathrm{d}s + \int_{t - \varepsilon}^t \mathcal{E}(s, \rho_s) \,\mathrm{d}s \\
   + \int_{\R^d \times \R^d} \langle \xi(x), (y-x)\rangle \gamma^\varepsilon(\,\mathrm{d}x , \,\mathrm{d}y) - \theta \varepsilon + o(\varepsilon) + o\left(\W_2(\mu, \rho_{t - \varepsilon})\right).
\end{multline*}
But, by definition of the metric derivative (\cite[Definition 1.1.2]{AMBROSIO}) 

\begin{equation*}
\lim_{\varepsilon \rightarrow 0} \frac{\W_2(\mu, \rho_{t-\varepsilon})}{\varepsilon} = |\rho_t'|(t) \: \text{ so } \: \lim_{\varepsilon \rightarrow 0} \frac{o\left(\W_2(\mu, \rho_{t-\varepsilon})\right)}{\varepsilon}  = 0.
\end{equation*}

By Young's inequality, we get that

\begin{align*}
    \int_{\R^d \times \R^d} \langle \xi(x), (x-y)\rangle \gamma^\varepsilon(\,\mathrm{d}x , \,\mathrm{d}y) & \geq - \frac{\varepsilon}{2} \| \xi \|_{L^2(\mu)}^2 - \frac{\W_2^2(\mu, \rho_{t - \varepsilon})}{2\varepsilon} \notag \\
    & \geq - \frac{\varepsilon}{2} \| \xi \|_{L^2(\mu)}^2 - \frac{1}{2} \int_{t- \varepsilon}^t \int_{\R^d}  |v(s,x)|^2 \,\mathrm{d}\rho_s \,\mathrm{d}s.
\end{align*}

Dividing by $\varepsilon$, we get

\begin{equation*}
    - \frac{1}{2}\| \xi \|_{L^2(\mu)}^2 - \theta + \frac{1}{\varepsilon} \int_{t - \varepsilon}^t \mathcal{E}(s, \rho_s) \,\mathrm{d}s + o(1) \leq 0. 
\end{equation*}

But since $s \rightarrow \rho_s$ is continuous, and since $\mathcal{E}$ is l.s.c. on $\R_+ \times \Pro $, 

\begin{equation*}
    - \frac{1}{2}\| \xi \|_{L^2(\mu)}^2 - \theta + \mathcal{E}(t, \rho_t) \leq 0,
\end{equation*}

which concludes the proof.
\end{itemize}
  
\end{proof}

\begin{remark}
    One may wonder how restrictive are the conditions we impose on the observation term $\mathcal{E}$. In fact, if $\bm{\rho}^{\text{obs}} \in AC^2([0, +\infty), \Pro)$, then $\mathcal{E}(t, \mu) = \frac{1}{2} \W_2^2(\mu, \rho^{\text{obs}}(t))$ satisfies the assumptions of the previous theorem. Indeed, since $\W_2^2$ is l.s.c. and since $\bm{\rho}^{\text{obs}}$ is absolutely continuous, the first two points are satisfied. We also have that

    \begin{equation*}
        \W_2^2(\mu, \rho^{\text{obs}}(t)) \leq 2 \int_{\R^d} |x|^2 \mu (\dr x) + 2 \int_{\R^d} |x|^2 \rho_t^{\text{obs}}(\dr x).
    \end{equation*}

    The only point which requires care is to verify that $t \rightarrow \mathcal{E}(t, (\id + (T - t)\nabla \phi)_\# \mu $ is upper semi-continuous for every $\phi \in C_c^\infty(\R^d)$ and $T>t$. Define $\mu_t = (F_t)_\# \mu $ for $F_t(x) = x + (T - t)\nabla \phi(x)$, and fix some $0 \leq t_0 < T$. Define $\gamma = ( F_{t_0} \times F_t)_\# \mu \in \Pi(\mu_{t_0}, \mu_t)$. We then have

    \begin{equation*}
        \W_2^2(\mu_t, \mu_{t_0}) \leq \int_{\R^d \times \R^d} |x - y|^2 \gamma(\dr x, \dr y) = \int_{\R^d} |F_t(x) - F_{t_0}(x)|^2 \,\mathrm{d}\mu(x) \leq  |t - t_0|^2 \| \nabla \phi \|_{\mathrm{L}^2(\R^d)}^2,
    \end{equation*}

    so that $\bm{\mu} = (\mu_t)_{t \in [0,T)} \in AC^2([0, T), \Pro)$ and then $t \rightarrow \W_2^2(\mu_t, \rho^{\text{obs}}(t))$ is in fact continuous.
\end{remark}

\subsection{Hilbertian approach and comparison principle}

In this section, we adapt the Hilbertian approach developed in \cite{LIONS} to our simple setting. We heavily rely on \cite{BERTUCCI} and \cite{TUDORASCU}, who developed a general framework to tackle Hamilton-Jacobi-Bellman equations in Wasserstein spaces. However, the fact that we consider trajectories with fixed endpoint instead of fixed starting point as in the usual optimal control point of view yields some extra technical difficulties which we tackle in what follows. Let us begin by redefining the notions of differentiability and of viscosity solutions, which are taken from \cite[Section 2]{BERTUCCI}, but in a more local version as this makes some difference when considering a non-compact space. Namely, we only ask some properties to be locally verified.

\begin{definition}[\cite{BERTUCCI}]
Let $\mathcal{F}: \Pro \rightarrow \R$ u.s.c., we say that $\psi : \R^d \rightarrow \Pro$ belongs to the superdifferential of $\mathcal{F}$ at $\mu$, and we note $\psi \in \partial^+\mathcal{F}(\mu)$, if:

\begin{enumerate}
    \item $x \rightarrow \int_{\R^d} z \psi(x, \dr z) \in \mathrm{L}^1(\mu)$,
    \item For every $\nu \in \Pro$, $\gamma \in \Pi(\mu, \nu)$, 
    
    \begin{equation*}
        \mathcal{F}(\nu) - \mathcal{F}(\mu) \leq \int_{\R^d} \int_{\R^d \times \R^d} z \cdot (y -x) \psi(x, \dr z) \gamma(\dr x , \dr y) + o\left( \left( \int_{\R^d \times \R^d} |x-y|^2 \gamma(\dr x, \dr y)\right)^\frac{1}{2} \right).
    \end{equation*}
\end{enumerate}

If $\mathcal{F}$ is l.s.c., we define its lower-differential as $\partial^- \mathcal{F} = \left\{ x \rightarrow (-\id)_\# \psi(x) \, | \, \psi \in \partial^+\mathcal{F}(\mu) \right\}$.
\end{definition}

We are now able to give the definition of viscosity solutions to the HJB equation \eqref{HJBW}.

\begin{definition}[\cite{BERTUCCI}] \label{VISCODEF}
We say that $\V : \R_+ \times \Pro \rightarrow \R$ is a viscosity subsolution of \eqref{HJBW} if $\V(0, \cdot) \leq \V_0$ and, for any $t \in \R_+, \mu \in \Pro$, and $(\theta, \psi) \in \partial^+\V(t,\mu)$,

\begin{equation*}
    \theta + \frac{1}{2} \int_{\R^d \times \R^d} |y|^2 \psi(x, \dr y) \mu(\dr x) - \mathcal{E}(t, \mu) \leq 0.
\end{equation*}

Similarly, we say that $\V : \R_+ \times \Pro \rightarrow \R$ is a viscosity supersolution of \eqref{HJBW} if $\V(0, \cdot) \geq \V_0$ and, for any $t \in \R_+, \mu \in \Pro$, and $(\theta, \psi) \in \partial^-\V(t,\mu)$,

\begin{equation*}
    \theta + \frac{1}{2} \int_{\R^d \times \R^d} |y|^2 \psi(x, \dr y) \mu(\dr x) - \mathcal{E}(t, \mu) \geq 0.
\end{equation*}

Finally, we say that $\V : \R_+ \times \Pro \rightarrow \R$ is a viscosity solution of \eqref{HJBW} if $\V$ is continuous and if $\V$ is both a sub and supersolution.
\end{definition}

We first prove a similar result to that in \cite[Proposition 2.4]{BERTUCCI}, by replacing the underlying space by $\R^d$ instead of $\mathbb{T}^d$.

\begin{lemma} \label{SUPERDIFFW}
Let $\mu, \nu \in \Pro$, $\gamma_o \in  \Pi_o(\mu, \nu)$ and $\phi : \mu' \in \Pro \rightarrow \frac{1}{2} \W_2^2(\mu', \nu)$. We denote by $\psi$ the measurable map defined almost everywhere by the disintegration $(\pi_1, \pi_1 - \pi_2)_{\#} \gamma_o = \mu(\dr x) \psi(x, \dr z)$, where $\pi_i$ is the projection on the $i$-th coordinate. Then $\psi \in \partial^+\phi(\mu)$.
\end{lemma}

\begin{proof}
Let $\gamma \in \Pi(\mu, \mu')$ be an arbitrary coupling, that we disintegrate following $\gamma(\dr x, \dr y) = k(x, \dr y) \mu( \dr x)$. Let $\gamma_o \in \Pi_o(\mu, \nu)$ and let us define $\pi(\dr y, \dr z) = \int_{\R^d} \gamma_o(\dr x, \dr y) k(x, \dr z) \in \Pi(\mu', \nu)$. It yields

\begin{align*}
    2\phi(\mu') - 2\phi(\mu) & \leq \int_{\R^d \times \R^d \times \R^d} |y-z|^2 k(x, \dr z)  \gamma_o(\dr x, \dr y) - \int_{\R^d \times \R^d} |x-z|^2 \gamma_o(\dr x, \dr y) \notag \\
    & \quad +\int_{\R^d \times \R^d \times \R^d} |y-x + x - z|^2 - |x-z|^2 k(x, \dr z) \gamma_o(\dr x, \dr y) \notag \\
    & = 2 \int_{\R^d \times \R^d \times \R^d} (y-x) \cdot (x-z) \gamma_o(\dr x, \dr y) k(x, \dr z) + \int_{\R^d \times \R^d} |x-y|^2 \gamma(\dr x, \dr y).
\end{align*}

The only thing left to prove is that $x \rightarrow \int_{\R^d} z \psi(x, \dr z) \in \mathrm{L}^1(\mu)$. But, by Jensen's inequality, 

\begin{equation*}
    \int_{\R^d} \left| \int_{\R^d} z \psi(x, \dr z) \right| \mu(\dr x) \leq \left( \int_{\R^d \times \R^d} |z|^2\psi(x, \dr z) \mu(\dr x)  \right)^\frac{1}{2} = \W_2(\mu, \nu) < +\infty.
\end{equation*}

\end{proof}

We also adapt \cite[Lemma 2.12]{BERTUCCI} to the case where the space is $\R^d$.

\begin{lemma} \label{LOCALMAX}
Let $\U$ be an upper semi-continuous function, $\Phi$ a continuous function and $(t, \mu) \in \R_+ \times \Pro$ such that $\U - \Phi$ has a maximum at $(t, \mu)$. Then $(\theta, \psi) \in \partial^+\Phi(t, \mu) \implies (\theta, \psi) \in \partial^+\U(t, \mu)$.
\end{lemma}

\begin{proof}
Let $t \in \R_+, \mu \in \Pro$ such that $\U - \Phi$ has a maximum at $(t, \mu)$ and let $(\theta, \psi) \in \partial^+ \Phi(t,\mu)$. One then has 

\begin{align*}
    \U(s, \mu') - \U(t, \mu) & \leq \Phi(s, \mu') - \Phi(t, \mu) \notag \\
    & \leq \theta(s - t) + \int_{\R^d \times \R^d \times \R^d} z \cdot (y-x) \psi(x, \dr z) \gamma(\dr x,\dr y) \notag \\
    & \quad + o\left(|t-s| + \left(\int_{\R^d \times \R^d} |x-y|^2 \gamma(\dr x, \dr y)\right)^\frac{1}{2} \right). 
\end{align*}
\end{proof}

We are now ready to state a comparison principle for \eqref{HJBW}, which was proven in \cite[Theorem 2.11]{BERTUCCI} on the torus $\mathcal{P}(\mathbb{T}^d)$. We adapt it to the non-compact space $\Pro$, which requires some adaptation of Stegall's lemma, which was used before in \cite{BERTUCCI2} and in \cite{BERT} to prove a comparison principle for another type of HJB equation.

\begin{theorem}\label{COMPPRINC}
Assume that $\mathcal{E}$ is Lipschitz, i.e. that there exists $C>0$ such that, for all $t,s \in \R_+, \mu, \nu \in \Pro$,

\begin{equation*}
    |\mathcal{E}(t, \mu) - \mathcal{E}(s, \nu)| \leq C\left( \left|t-s\right| + \W_2(\mu, \nu)\right).
\end{equation*}

Let $\mathcal{U}$ be a subsolution to \eqref{HJBW}, and $\V$ a supersolution. Further assume that $\mathcal{U}(0,\cdot) \leq \V(0, \cdot)$. Then, for all $t\geq 0, \mu \in \Pro$, $\mathcal{U}(t,\mu) \leq \V(t,\mu)$. 
\end{theorem}

\begin{proof}
We argue by contradiction, and therefore assume that there exists $T>0, M>0$ such that

\begin{equation*}
    \inf_{t \leq T, \mu \in \Pro} \V(t,\mu) - \U(t,\mu) \leq - M.
\end{equation*}

Let us proceed by using the doubling of variables method, which was first introduced for the finite dimensional case. We have $R>0$ such that, for every $\phi, \psi \in C^2(\R^d)$, $\| \phi \|_{C^2} + \| \psi\|_{C^2} \leq h$, $\varepsilon > 0$, 

\begin{equation*}
    \inf_{t,s \leq T, \mu, \nu \in \Pro} \V(s, \nu) - \U(t, \mu) + \langle \phi, \mu \rangle + \langle \psi, \nu \rangle+ \frac{1}{2 \varepsilon} \left( (t-s)^2 + \W_2^2(\mu, \nu) \right) +R(t+s) \leq - \frac{M}{2}.
\end{equation*}

The variation of Stegall's lemma in \cite[Lemma 2.1.]{BERTUCCI2} gives us the existence of a strict minimum $(t_\varepsilon, s_\varepsilon, \mu_\varepsilon, \nu_\varepsilon)$. First assume that $t_\varepsilon, s_\varepsilon >0$. Let us take $\gamma_\varepsilon \in \Pi_o(\mu_\varepsilon, \nu_\varepsilon)$ an optimal coupling between $\mu_\varepsilon$ and $\nu_\varepsilon$, and $\psi_\varepsilon$ defined by $\left(\pi_1, \phi \circ \pi_1 - \psi \circ \pi_2 + \frac{1}{\varepsilon} \left( \pi_1 - \pi_2 \right) \right)_\# \gamma_\varepsilon (\dr x, \dr z) = \mu_\varepsilon (\dr x) \psi_\varepsilon(x, \dr z)$. Thanks to \autoref{SUPERDIFFW} and \autoref{LOCALMAX}, we get

\begin{equation*}
    \left(R + \frac{1}{\varepsilon}(t_\varepsilon - s_\varepsilon), \psi_\varepsilon\right) \in \partial^+\U(t_\varepsilon, \mu_\varepsilon),
\end{equation*}

and

\begin{equation*}
    \left(-R - \frac{1}{\varepsilon}(s_\varepsilon - t_\varepsilon), \psi_\varepsilon\right) \in \partial^- \V(s_\varepsilon, \nu_\varepsilon).
\end{equation*}

Since $\U$ and $\V$ are respectively viscosity sub and supersolutions, we then have that 

\begin{equation*}
    R + \frac{1}{\varepsilon}(t_\varepsilon - s_\varepsilon) + \frac{1}{2\varepsilon^2} \W_2^2(\mu_\varepsilon, \nu_\varepsilon) - \mathcal{E}(t_\varepsilon, \mu_\varepsilon) + \langle \nabla \phi, \mu_\varepsilon \rangle  \leq 0,
\end{equation*}

and 

\begin{equation*}
    - R - \frac{1}{\varepsilon}(s_\varepsilon - t_\varepsilon) + \frac{1}{2\varepsilon^2} \W_2^2(\mu_\varepsilon, \nu_\varepsilon) - \mathcal{E}(s_\varepsilon, \nu_\varepsilon) -\langle \nabla \psi, \nu_\varepsilon\rangle \geq 0.
\end{equation*}

Subtracting those two inequalities yields

\begin{equation*}
    2R  \leq \mathcal{E}(s_\varepsilon, \nu_\varepsilon) - \mathcal{E}(t_\varepsilon, \mu_\varepsilon) + \langle \nabla \phi, \mu_\varepsilon \rangle + \langle \nabla \psi, \nu_\varepsilon \rangle\leq C \left( |t_\varepsilon - s_\varepsilon| + \W_2(\mu_\varepsilon, \nu_\varepsilon) + h \right).
\end{equation*}

But $|t_\varepsilon - s_\varepsilon| + \W_2(\mu_\varepsilon, \nu_\varepsilon) \rightarrow 0$ so that we get a contradiction, letting $\varepsilon \rightarrow 0$ and $h \rightarrow 0$, as $R>0$. When $s_\varepsilon =0$ or $t_\varepsilon = 0$, the proof is exactly the same as in  \cite[Theorem 2.11]{BERTUCCI}, which concludes the proof.
\end{proof}

\subsection{Viscosity properties of the value function}

We have already seen that our value function is continuous for the $\W_2$ topology, hence we only have to prove that the inequalities related to the viscosity properties are still valid in this new setting, which is an adaptation of \cite[Section 4.3]{BERTUCCI}, in the case where the Hamiltonian depends also on $\mu$ and where time is reversed. We highlight the fact that we only prove the subsolution property for elements of the superdifferential of the form $\mu(\dr x)\psi(x, \dr z) = (\pi, \lambda(\pi_1 - \pi_2))_\# \gamma_o(\dr x, \dr z)$, which is enough in order to have a comparison principle.

\begin{theorem}\label{VISCOSOLB}
Assume that $\mathcal{E}$ is Lipschitz continuous with respect to $t$ and $\mu$. Then $\V$ is a viscosity supersolution to \eqref{HJBW} in the sense of \autoref{VISCODEF}. Furthermore, if $\lambda>0, \nu \in \Pro$, $(\theta, \psi) \in \partial^+\V(t, \mu)$ such that, if $\gamma_o \in \Pi_o(\mu, \nu)$, 
\begin{equation*}
    \mu(\dr x) \psi(x, \dr z) = (\pi_1, \lambda(\pi_1 - \pi_2))_\# \gamma_o(\dr x, \dr z),
\end{equation*}

then,

\begin{equation*}
        \theta + \frac{1}{2}\int_{\R^d \times \R^d} |z|^2 \psi(x, \dr z) \mu(\dr x) - \mathcal{E}(t, \mu) \leq 0.
\end{equation*}

\end{theorem}

\begin{proof}
Thanks to \autoref{VISCOSOLW}, we know that 

\begin{equation*}
    \theta + \frac{1}{2} \| \xi \|_{\mathrm{L}^2(\mu)}^2 -  \mathcal{E}(t,\mu) \geq 0, \quad\forall (t, \mu) \in \R_+ \times \Pro, \; (\theta, \xi) \in \partial_. \mathcal{V}(t, \mu).
\end{equation*}

where we recall that $\partial_. \mathcal{V}(t, \mu)$ denotes the classical subdifferential of $\mathcal{V}$ at $(t, \mu)$. In particular, taking $(\theta, \psi) \in \partial^-\V(t,\mu)$ yields

\begin{equation*}
    \theta + \frac{1}{2} \int_{\R^d} \left|\int_{\R^d} z \psi(x,\dr z)\right|^2 \mu(\dr x) - \mathcal{E}(t, \mu) \geq 0,
\end{equation*}

which gives the supersolution property thanks to Jensen's inequality. Indeed, we have that

\begin{equation*}
    \int_{\R^d} \left|\int_{\R^d} z \psi(x,\dr z)\right|^2 \mu(\dr x) \leq \int_{\R^d} \int_{\R^d} |z|^2 \psi(x, \dr z) \mu(\dr x),
\end{equation*}

so that the inequality is satisfied. We now prove the second part of the statement, which is more technical. Let us take $s>0$, $z \in \text{Supp}(\rho_s)$, where $\rho_s = \mathcal{L}((1- \lambda(t-s)) X + \lambda (t-s) Y)$ and $\gamma_o = \mathcal{L}(X,Y)$ is an optimal coupling between $\mu$ and $\nu$. We are going to show that there exists at most one couple $(x,y) \in \text{Supp}(\gamma_o)$ such that $z = (1-  \lambda(t-s) )  x +  \lambda (t-s) y $, at least for $s$ small enough. Let $(x,y), (x', y')$ be two such couples. We then have 

\begin{equation*}
    x-x' = \frac{ \lambda(t-s)}{1 - \lambda(t-s)} (y'- y),
\end{equation*}

and, since $(x,y), (x',y') \in \text{Supp}(\gamma_o)$,

\begin{equation*}
    0 \leq (x - x') \cdot (y - y') = - \frac{ \lambda (t-s)}{1 -  \lambda (t-s)} |y'- y|^2 \leq 0,
\end{equation*}

which gives $y = y'$, and then $x = x'$, at least for $t-s$ smaller than $\frac{1}{\lambda}$. We can now define two measurable maps $X_s$ and $Y_s$ to be such that

\begin{equation*}
    z = (1 - \lambda (t - s)) X_s(z) + \lambda  (t-s) Y_s(z),
\end{equation*}

for every $z \in \text{Supp}(\rho_s)$. In particular, for $s$ close enough to $t$, $(\bm{\rho}, v)$ solves the continuity equation, for $v(s,z) = \lambda( X_s(z) -  Y_s(z) )$, and satisfies $\rho_t = \mu$. Then, thanks to \autoref{DDP}, we get, for $\delta >0$ small enough,

\begin{align*}
    \V(t, \mu) & \leq \frac{1}{2} \int_{t - \delta}^t \int_{\R^d}  \lambda^2 |X_s(z) -  Y_s(z) |^2  \dr \rho_s \dr s + \int_{t - \delta}^t \mathcal{E}(s, \rho_s) \dr s + \V(t - \delta, \rho_{t- \delta}) \notag \\
    & = \frac{\delta \lambda^2 }{2} \int_{\R^d} |x -  y |^2  \gamma_o(\dr x, \dr y) + \int_{t - \delta}^t \mathcal{E}(s, \rho_s) \dr s + \V(t - \delta, \rho_{t- \delta}).
\end{align*}

But, since $(\theta, \psi) \in \partial^+\V(t, \mu)$, for $\delta$ small enough,

\begin{equation*}
    \V(t, \mu) - \V(t - \delta, \rho_{t- \delta}) \geq   \lambda^2 \delta \int_{\R^d}  |x -  y |^2 \gamma_o(\dr x, \dr y)  + \theta \delta + o\left( \delta  \right).
\end{equation*}

This yields

\begin{equation*}
    0 \leq - \theta \delta - \delta \frac{\lambda^2}{2}  \int_{\R^d}  |x -  y |^2 \gamma_o(\dr x, \dr y) + \int_{t - \delta}^t \mathcal{E}(s, \rho_s) \dr s.
\end{equation*}

As $\mathcal{E}$ is Lipschitz continuous, dividing by $\delta$ and letting $\delta\rightarrow 0$ gives the result we were looking for.
\end{proof}

\section{Extension to a transport equation}

In this section, we briefly show how our framework can be adapted in order to tackle the transport equation.  We indeed wish to extend our approach to more intricate dynamics. The previous section can be understood as some state estimation for a steady state --~ since $\rho$ was not subjected to any dynamics~--. In this section, we show that it can be adapted to more complex dynamics, namely transport equation with a Lipschitz drift. We consider

\begin{equation}
    \partial_t \rho(t,x) + \nabla \cdot (\rho (t,x)b(t,x)) + \nabla \cdot (\rho (t,x)v(t,x)) =0,
    \label{TRANS}
\end{equation}

where we assume that $b$ satisfies

\begin{equation*}
    |b(t,x) - b(s,y)| \leq C \left( |t-s| + |x-y|\right),
\end{equation*}

so that \eqref{TRANS} is well-posed. Our value function is then given by

\begin{equation}
    \mathcal{V}(t,\rho) = \inf_{\substack{\bm{\rho} \in AC^2(0,t, \Pro) \\ \rho_t = \rho}} \left\{\mathcal{V}_0(\rho_0) + \int_0^t \int_{\mathbb{R}^d} \frac{1}{2} |v(s,x) - b(s,x)|^2 \,\mathrm{d}\rho_s \,\mathrm{d}s + \int_0^t \mathcal{E}(s, \rho_s) \,\mathrm{d}s\right\}.
    \label{VALUETRANS}
\end{equation}

From now on, we assume that

\begin{equation*}
        |\mathcal{E}(t, \mu) - \mathcal{E}(s, \nu)| \leq C\left( (t-s) + \W_2(\mu, \nu)\right),
\end{equation*}

and that $\V_0$ is continuous and non-negative. One could assume that $\V_0$ is only l.s.c. as before but we choose here to assume stronger regularity in order to shorten the proof of the continuity of the value function.

\begin{remark}
We could equivalently consider

\begin{equation*}
    \mathcal{W}(t,\rho) = \inf_{\substack{\bm{\rho} \in AC^2(0,t,\Pro) \\ \rho_t = \rho}} \left\{\mathcal{V}_0(\rho_0) + \int_0^t \int_{\mathbb{R}^d} \frac{1}{2} |v(s,x)|^2 \,\mathrm{d}\rho_s \,\mathrm{d}s + \int_0^t \mathcal{E}(s, \rho_s) \,\mathrm{d}s\right\},
\end{equation*}

where the infimum is taken over curves satisfying \eqref{TRANS}.
\end{remark}

\subsection{Continuity of the value function}

By adapting the proofs of the previous section, we are able to prove the following results. Let us highlight the fact that most of the results are not altered by this more complex setting, the only difference lies in the technicalities of the proofs.

\begin{lemma}\label{DDPTRANS}
The value function $\V$ satisfies, for $0 \leq s < t$,

\begin{equation*}
    \V(t, \rho) = \inf_{\substack{\bm{\rho} \in AC^2(s,t, \Pro) \\ \rho_t = \rho}} \left\{  \V(s, \rho_s ) + \int_s^t \int_{\R^d}  \frac{1}{2} |v(\tau,x) - b(\tau, x)|^2 \,\mathrm{d}\rho_\tau \,\mathrm{d}\tau + \int_s^t \mathcal{E}(\tau, \rho_\tau)  \,\mathrm{d}\tau \right\}.
\end{equation*}
\end{lemma}

\begin{proof}
    The proof follows the exact same approach as in \autoref{DDP}.
\end{proof}

We also prove the continuity of the value function. We first rely on an adaptation of \cite[Lemma 4.1]{BERTUCCI}, in order to reduce the set of admissible curves. However, the fact that the Lagrangian also depends on time and does not satisfy the exact same assumptions of this lemma yield extra difficulties.

\begin{lemma}
Let us define, for $0 \leq s < t$,

\begin{equation*}
    A^\varepsilon(t, \mu) = \left\{ \bm{\rho} \in AC^2(0,t, \Pro) \, , \, \forall s \in [t-\varepsilon, t], \, \rho_s = \mu \right\}.
\end{equation*}

Then

\begin{equation*}
    \V(t, \mu) = \inf_{\varepsilon>0, \bm{\rho} \in  A^\varepsilon(t, \mu)} \left\{ \mathcal{V}_0(\rho_0) + \int_0^t \int_{\mathbb{R}^d} \frac{1}{2} |v(s,x) - b(s,x)|^2 \,\mathrm{d}\rho_s \,\mathrm{d}s + \int_0^t \mathcal{E}(s, \rho_s) \,\mathrm{d}s \right\},
\end{equation*}

and the infimum is reached for some $\bm{\rho} \in A^\varepsilon(t, \mu)$.
\end{lemma}

\begin{proof}
We only have one inequality to prove. Let $n \geq 1$, $\varepsilon>0$ and $\bm{\rho}$ such that 

\begin{equation*}
   \mathcal{V}_0(\rho_0) + \int_0^t \int_{\mathbb{R}^d} \frac{1}{2} |v(s,x) - b(s,x)|^2 \,\mathrm{d}\rho_s \,\mathrm{d}s + \int_0^t \mathcal{E}(s, \rho_s) \,\mathrm{d}s \leq \V(t, \mu) + \frac{1}{n}.
\end{equation*}

Define $\tilde{\rho}(s) = \rho\left( \frac{ts}{t -\varepsilon} \right)$, so that $\tilde{\rho}(t - \varepsilon) = \mu$ and $\tilde{\rho}(0) = \rho(0)$. We extend $\tilde{\rho}$ to $[t-\varepsilon,t]$ by setting $\tilde{\rho} = \mu$ on the subinterval, so that $\tilde{\rho} \in A^\varepsilon(t, \mu)$. We then get

\begin{multline*}
    \mathcal{V}_0(\tilde{\rho}_0) + \int_0^t \int_{\mathbb{R}^d} \frac{1}{2} |\tilde{v}(s,x) - b(s,x)|^2 \,\mathrm{d}\tilde{\rho}_s \,\mathrm{d}s + \int_0^t \mathcal{E}(s, \tilde{\rho}_s) \,\mathrm{d}s 
    = \mathcal{V}_0(\rho_0) 
    + \int_0^{t - \varepsilon} \int_{\mathbb{R}^d} \frac{1}{2} |\tilde{v}(s,x) - b(s,x)|^2 \,\mathrm{d}\tilde{\rho}_s \,\mathrm{d}s \\
    + \int_0^{t-\varepsilon} \mathcal{E}(s, \tilde{\rho}_s) \,\mathrm{d}s + \int_{t - \varepsilon}^t \int_{\R^d} \frac{1}{2} |b(s,x)|^2 \mu(\dr x) \dr s + \int_{t- \varepsilon}^t \mathcal{E}(s, \mu) \dr s.
\end{multline*}

Since $b$ and $\mathcal{E}$ are Lipschitz continuous, there exists $C>0$ such that

\begin{multline*}
  \frac{1}{2}  \int_{t - \varepsilon}^t \int_{\R^d} |b(s,x)|^2 \mu(\dr x) \dr s + \int_{t- \varepsilon}^t \mathcal{E}(s, \mu) \dr s \\
    \leq \int_{t-\varepsilon}^t \int_{\R^d}  \left( 2C(s^2 +|x|^2) + 2|b(0,0)|^2 \right) \mu(\dr x) + \int_{t-\varepsilon}^t s + |\mathcal{E}(0,0)| \dr s \\
    \leq \left( 2 |b(0,0)|^2 + |\mathcal{E}(0,0)| + 2C \int_{\R^d} |x|^2 \mu(\dr x) \right) \varepsilon + o(\varepsilon).
\end{multline*}

By performing a change of variable on the other integral, we get

\begin{multline*}
    \int_0^{t - \varepsilon} \int_{\mathbb{R}^d} \frac{1}{2} |\tilde{v}(s,x) - b(s,x)|^2 \,\mathrm{d}\tilde{\rho}_s \,\mathrm{d}s 
     + \int_0^{t-\varepsilon} \mathcal{E}(s, \tilde{\rho}_s) \,\mathrm{d}s \\
    = \frac{t -\varepsilon}{t} \left( \int_0^t \int_{\R^d} \frac{1}{2} \left| \frac{t - \varepsilon}{t} v(\tau,x) - b\left(\frac{t - \varepsilon}{t} \tau, x\right) \right|^2 \dr \rho_\tau \dr \tau + \int_0^t \mathcal{E}\left(\frac{t - \varepsilon}{t}\tau, \rho_\tau\right) \dr \tau \right).
\end{multline*}
Since $\mathcal{E}$ is Lipschitz continuous, we get that 

\begin{equation*}
    \int_0^t \mathcal{E}\left(\frac{t - \varepsilon}{t}\tau, \rho_\tau\right) \dr \tau \leq \int_0^t \mathcal{E}(\tau, \rho_\tau) \dr \tau + \int_0^t C\left(1 - \frac{t-\varepsilon}{t}\right)\tau \dr \tau = \int_0^t \mathcal{E}(\tau, \rho_\tau) \dr \tau + O(\varepsilon).
\end{equation*}

Lastly, since $b$ is Lipschitz continuous and using Cauchy-Schwarz,

\begin{multline*}
    \int_0^t \int_{\R^d} \left| \frac{t - \varepsilon}{t} v(\tau,x) - b\left(\frac{t - \varepsilon}{t} \tau, x\right) \right|^2 - \left(\frac{t - \varepsilon}{t} \right)^2 |v(\tau,x) - b(\tau,x)|^2 \dr \rho_\tau \dr \tau \\
    = \int_0^t \int_{\R^d} \left|b\left(\frac{t-\varepsilon}{t} \tau, x\right)\right|^2 - \left(\frac{t-\varepsilon}{t}\right)^2 |b(\tau,x)|^2 + 2 \frac{t-\varepsilon}{t} v(\tau,x) \cdot \left(\frac{t-\varepsilon}{t} b(\tau,x) - b\left(\frac{t - \varepsilon}{t}\tau,x\right)\right) \dr \rho_\tau \dr \tau \\
    \leq \int_0^t \int_{\R^d} \left|b\left(\frac{t-\varepsilon}{t} \tau, x\right)\right|^2 - \left(\frac{t-\varepsilon}{t}\right)^2 |b(\tau,x)|^2 + 2 C \frac{t-\varepsilon}{t} |v(\tau,x)| \left( 1 - \frac{t-\varepsilon}{t}\right) \dr \rho_\tau \dr \tau = O(\varepsilon),
\end{multline*}

by the dominated convergence theorem, as $b$ is Lipschitz. This gives that

\begin{equation*}
    \mathcal{V}_0(\tilde{\rho}_0) + \int_0^t \int_{\mathbb{R}^d} \frac{1}{2} |\tilde{v}(s,x) - b(s,x)|^2 \,\mathrm{d}\tilde{\rho}_s \,\mathrm{d}s + \int_0^t \mathcal{E}(s, \tilde{\rho}_s) \,\mathrm{d}s  \leq \frac{t - \varepsilon}{t}\left( \V(t, \mu) + \frac{1}{n} \right) + O(\varepsilon),
\end{equation*}

which is exactly the type of bound we were looking for.
\end{proof}

\begin{remark}
    This lemma can be interpreted as an adaptation of \cite[Lemma 3.2]{BARAS} to our setting.
\end{remark}

\begin{prop}
$\V$ is continuous on $\R_+ \times \Pro$ for the topology induced by the $\W_2$ distance. 
\end{prop}

\begin{proof}
Let $\mu, \nu \in \Pro$, $t>0$ and $\varepsilon >0$. Let us take $\bm{\rho}$ given by the previous lemma (i.e. such that $\rho_s = \mu$ for $s \in [0,t- K\varepsilon ]$ and being an $\varepsilon$-optimal curve). We extend this curve by taking $\bm{\rho}$ to be an optimal curve for the Wasserstein distance $\W_2$ between $\mu$ and $\nu$ in time $\varepsilon$ (i.e. minimizing the Benamou-Brenier formulation for a time $\varepsilon$ instead of $1$). Then,

\begin{equation*}
    \V(t, \nu) \leq \V(t, \mu) + \varepsilon + \int_{t-\varepsilon}^t \int_{\R^d} |v(s,x) - b(s,x)|^2 \dr \rho_s \dr s + \int_{t- \varepsilon}^t \mathcal{E}(s,\rho_s) \dr s.
\end{equation*}

Recall that, using Young's inequality and the fact that $b$ is Lipschitz, 

\begin{equation*}
    \int_{t-\varepsilon}^t \int_{\R^d} |v(s,x) - b(s,x)|^2 \dr \rho_s \dr s \leq \frac{2 \W_2^2(\mu, \nu)}{\varepsilon} + 2\varepsilon \int_{t- \varepsilon}^t \int_{\R^d} |x|^2 \dr \rho_s \dr s + O(\varepsilon).
\end{equation*}

But, by the triangular inequality,

\begin{equation*}
    \int_{\R^d} |x|^2 \dr \rho_s \leq 2 \int_{\R^d} |x|^2 \dr \mu + 2 \W_2^2(\mu, \rho_s) \leq 2\int_{\R^d} |x|^2 \dr \mu  + \frac{2 \W_2^2(\mu, \rho_s)}{\varepsilon^2},
\end{equation*}

so that 

\begin{equation*}
    \int_{t- \varepsilon}^t \int_{\R^d} |x|^2 \dr \rho_s \dr s \leq 2 \varepsilon \int_{\R^d} |x|^2 \dr \mu +  \frac{2 \W_2^2(\mu, \nu)}{\varepsilon}.
\end{equation*}

Similarly, as $\mathcal{E}$ is Lipschitz, 

\begin{equation*}
    \int_{t- \varepsilon}^t \mathcal{E}(t,\rho_s) \dr s \leq \varepsilon \mathcal{E}(0, \mu) + C\int_{t - \varepsilon}^t s + \W_2(\mu, \rho_s) \dr s \leq  C \W_2(\mu, \nu)  + O(\varepsilon).
\end{equation*}

In the end, we obtained

\begin{equation} \label{PROBLIP}
    \V(t, \nu) \leq \V(t, \mu) + \frac{4 \W_2^2(\mu, \nu)}{\varepsilon} + C \W_2(\mu, \nu) +  O(\varepsilon).
\end{equation}

Taking $\varepsilon = \W_2(\mu, \nu)$ gives that $\V(t, \cdot)$ is continuous, as $\mu$ and $\nu$ are arbitrary in $\Pro$. Similarly, if $\varepsilon>0$, by Bellman principle and taking $v = 0$ between $t - \varepsilon>0$ and $t$,

\begin{equation*}
    \V(t, \mu) \leq \V(t - \varepsilon, \mu) + \int_{t- \varepsilon}^t \int_{\R^d}\frac{1}{2} |b(s,x)|^2 \dr \mu \dr s + \int_{t- \varepsilon}^t \mathcal{E}(s, \mu) \dr s.
\end{equation*}

Using the same type of argument as before, one gets

\begin{equation*}
    \V(t, \mu) \leq \V(t - \varepsilon, \mu) + O(\varepsilon).
\end{equation*}

Also observe that, by taking an $\varepsilon$-optimal control from the previous lemma,

\begin{equation*}
    \V(t, \mu) \leq \V(t - \varepsilon, \mu) + \int_{t-\varepsilon}^t \int_{\mathbb{R}^d} \frac{1}{2} |v(s,x) - b(s,x)|^2 \,\mathrm{d}\rho_s \,\mathrm{d}s + \int_{t - \varepsilon}^t \mathcal{E}(s, \rho_s) \,\mathrm{d}s \leq \V(t, \mu) + K\varepsilon.
\end{equation*}

This yields 

\begin{equation} \label{TIMELIP}
    \V(t, \mu) - \V(t - \varepsilon, \mu) \geq - C \varepsilon,
\end{equation}

Hence the result we were looking for, by putting together \eqref{TIMELIP} and \eqref{PROBLIP}.
\end{proof}

\subsection{Derivation of a first order HJB equation}

We will now prove that $\V$ is a viscosity solution to 

\begin{equation} \label{HJBTRANS}
\begin{dcases}
    \partial_t \V + \left\langle b(t, \cdot), \nabla_{\mu} \V(t, \mu) \right\rangle_{\mathrm{L}^2(\mu)} + \frac{1}{2} \left\| \nabla_\mu \V(t, \mu) \right\|^2_{\mathrm{L}^2(\mu)} = \mathcal{E}(t,\mu), \\
    \V(0, \mu) = \V_0(\mu).
\end{dcases}
\end{equation}

We first recall the definition of viscosity solutions.

\begin{definition} \label{VISCODEFTRANS}
We say that $\V : \R_+ \times \Pro \rightarrow \R$ is a viscosity subsolution of \eqref{HJBTRANS} if $\V(0, \cdot) \leq \V_0$ and, for any $t \in \R_+, \mu \in \Pro$, and $(\theta, \psi) \in \partial^+\V(t,\mu)$,

\begin{equation*}
    \theta + \int_{\R^d \times \R^d} z \cdot b(t,x) \psi(x, \dr z) \mu(\dr x) + \frac{1}{2} \int_{\R^d \times \R^d} |z|^2 \psi(x, \dr z) \mu(\dr x) - \mathcal{E}(t, \mu) \leq 0.
\end{equation*}

Similarly, we say that $\V : \R_+ \times \Pro \rightarrow \R$ is a viscosity supersolution of \eqref{HJBTRANS} if $\V(0, \cdot) \geq \V_0$ and, for any $t \in \R_+, \mu \in \Pro$, and $(\theta, \psi) \in \partial^-\V(t,\mu)$,

\begin{equation*}
    \theta + \int_{\R^d \times \R^d} z \cdot b(t,x) \psi(x, \dr z) \mu(\dr x) +  \frac{1}{2} \int_{\R^d \times \R^d} |z|^2 \psi(x, \dr z) \mu(\dr x) - \mathcal{E}(t, \mu) \geq 0.
\end{equation*}

Finally, we say that $\V : \R_+ \times \Pro \rightarrow \R$ is a viscosity solution of \eqref{HJBTRANS} if $\V$ is continuous and if $\V$ is both a sub and supersolution.
\end{definition}

We can also prove a comparison principle for this equation.

\begin{theorem}
Assume that $\mathcal{E}$ and $b$ are Lipschitz. Let $\mathcal{U}$ be a subsolution to \eqref{HJBTRANS} , $\V$ a supersolution. Further assume that $\mathcal{U}(0,\cdot) \leq \V(0, \cdot)$. Then, for all $t\geq 0, \mu \in \Pro$, $\mathcal{U}(t,\mu) \leq \V(t,\mu)$. 
\end{theorem}

\begin{proof}
The proof is essentially the same as in \autoref{COMPPRINC}, so we only highlight the difference. The inequalities we obtain are given by, keeping the same notations,

\begin{equation*}
        R + \frac{1}{\varepsilon}(t_\varepsilon - s_\varepsilon) + \int_{\R^d \times \R^d} \left( \frac{x-y}{\varepsilon} \right) \cdot b(t_\varepsilon,x) \gamma_\varepsilon(\dr x, \dr y) + \frac{1}{2\varepsilon^2} \W_2^2(\mu_\varepsilon, \nu_\varepsilon) - \mathcal{E}(t_\varepsilon, \mu_\varepsilon) + \langle \nabla \phi, \mu_\varepsilon\rangle\leq 0,
\end{equation*}
and 

\begin{equation*}
        - R - \frac{1}{\varepsilon}(s_\varepsilon - t_\varepsilon) + \int_{\R^d \times \R^d} \left( \frac{x-y}{\varepsilon} \right) \cdot b(s_\varepsilon,y) \gamma_\varepsilon(\dr x, \dr y) + \frac{1}{2\varepsilon^2} \W_2^2(\mu_\varepsilon, \nu_\varepsilon) - \mathcal{E}(s_\varepsilon, \nu_\varepsilon) - \langle \nabla \psi, \nu_\varepsilon \rangle \geq 0.
\end{equation*}

It yields

\begin{align*}
    2R &\leq \int_{\R^d \times \R^d} \left( \frac{x-y}{\varepsilon} \right) \cdot \left( b(t_\varepsilon,x) - b(s_\varepsilon,y) \right) \gamma_\varepsilon(\dr x, \dr y) + \mathcal{E}(s_\varepsilon, \nu_\varepsilon) - \mathcal{E}(t_\varepsilon, \mu_\varepsilon) + \langle \nabla \phi, \mu_\varepsilon \rangle + \langle \nabla \psi, \nu_\varepsilon \rangle \\
    &\leq C(|t_\varepsilon - s_\varepsilon| + \W_2(\mu_\varepsilon, \nu_\varepsilon)) + C\left( \frac{\W_2^2(\mu_\varepsilon, \nu_\varepsilon)}{\varepsilon} + \frac{|t_\varepsilon - s_\varepsilon|}{\varepsilon} \int_{\R^d \times \R^d} |x-y| \gamma_\varepsilon(\dr x, \dr y) + h\right).
\end{align*}

But $\frac{\W_2^2(\mu_\varepsilon, \nu_\varepsilon)}{\varepsilon} \rightarrow 0$ and $\frac{(t_\varepsilon - s_\varepsilon)^2}{\varepsilon} \rightarrow 0$ as $\varepsilon \rightarrow 0$, and $h$ is arbitrary small, hence we can continue the proof as before.
\end{proof}

 In order to derive the viscosity properties of the value function, we first need to adapt the supersolution part of our \autoref{VISCOSOLW}, which we prove thereafter.
 
\begin{lemma}\label{VISCOSOLTRANSW}
Assume that $\mathcal{E}$ and $b$ are Lipschitz, that $\V_0$ is bounded below and l.s.c. Then $\V$ is a classical viscosity supersolution to the Hamilton-Jacobi-Bellman equation \eqref{HJBTRANS} in the sense of \autoref{GEOMVIS}.
\end{lemma}

\begin{proof}
Let $\varepsilon, t>0,\mu \in \Pro$, $(\theta, \xi) \in \partial_{.}\V(t, \mu)$ and $\bm{\rho} \in AC^2(t-\varepsilon,t,\Pro)$ and $\rho_t = \mu$ satisfying $\partial_s \rho_s + \nabla \cdot (b(s, \cdot) \rho_s) = 0$ on $[t-\varepsilon, t]$ -- this is possible since $b$ is regular enough so that the continuity equation can be reversed in time --. Let $\gamma_\varepsilon \in \Pi(\rho_{t-\varepsilon}, \mu)$.  By Bellman principle and the definition of $\partial_{.}\V(t, \mu)$, we get

\begin{equation} \label{SUPDIFFIN}
  \int_{\R^d \times \R^d} \langle \xi(x), (y-x)\rangle \gamma^\varepsilon(\,\mathrm{d}x , \,\mathrm{d}y) - \theta \varepsilon + o(\varepsilon) + o\left(\W_2(\mu, \rho_{t - \varepsilon})\right) \\
 + \int_{t - \varepsilon}^t \mathcal{E}(s, \rho_s) \,\mathrm{d}s \leq 0.
\end{equation}

Since $\bm{\rho}$ solves the continuity equation with a regular enough velocity field, one gets, as $\varepsilon \rightarrow 0$,

\begin{equation*}
    \frac{1}{\varepsilon} \int_{\R^d \times \R^d} \langle \xi(x), (y-x)\rangle \gamma^\varepsilon(\,\mathrm{d}x , \,\mathrm{d}y) \rightarrow - \int_{\R^d} \langle \xi(x), b(t,x) \rangle \dr \mu.
\end{equation*}

We also observe that

\begin{equation*}
    \W_2^2(\rho_s,\rho_t) \leq \int_s^t \int_{\R^d} |b(\tau, x)|^2 \dr \rho_\tau \dr s \leq C \int_s^t (\tau^2 + \int_{\R^d} |x|^2 \dr \rho_\tau + |b(0,0)|^2) \dr \tau .
\end{equation*}

In particular, dividing by $(t-s)$ yields that $o(\W_2(\mu, \rho_{t- \varepsilon})) = o(\varepsilon)$ as $\varepsilon$ goes to 0. In the end, dividing \eqref{SUPDIFFIN} by $\varepsilon$ and letting $\varepsilon \rightarrow 0$ gives 

\begin{equation*}
- \frac{1}{2} \left\| \xi \right\|^2_{\mathrm{L}^2(\mu)} - \langle \xi, b(t, \cdot) \rangle_{\mathrm{L}^2(\mu)} - \theta  \leq - \langle \xi, b(t, \cdot) \rangle_{\mathrm{L}^2(\mu)} - \theta + \mathcal{E}(t, \mu) \leq 0,
\end{equation*}

which is what we wanted.
\end{proof}

We are now ready to prove that $\V$ is a solution to \eqref{HJBTRANS}, in the same sense as \autoref{VISCODEFTRANS}. This time, the time-dependence of $b$ yields extra difficulties that we manage to tackle.

\begin{theorem}
Assume that $b$ and $\mathcal{E}$ are Lipschitz continuous. Then $\V$ is a viscosity supersolution to \eqref{HJBTRANS} in the sense of \autoref{VISCODEFTRANS}. Furthermore, if $\lambda>0, \nu \in \Pro$ and $(\theta, \psi) \in \partial^+\V(t, \mu)$ are such that
\begin{equation*}
    \mu(\dr x) \psi(x, \dr z) = (\pi, \lambda(\pi_1 - \pi_2))_\# \gamma_o(\dr x, \dr z),
\end{equation*}

for $\gamma_o \in \Pi_o(\mu, \nu)$, then,

\begin{equation*}
        \theta +\int_{\R^d \times \R^d} z \cdot b(t,x) \psi(x, \dr z) \mu(\dr x) + \frac{1}{2}\int_{\R^d \times \R^d} |z|^2 \psi(x, \dr z) \mu(\dr x) - \mathcal{E}(t, \mu) \leq 0.
\end{equation*}

\end{theorem}

\begin{proof}
We proceed as in \autoref{VISCOSOLB}. Thanks to \autoref{VISCOSOLTRANSW}, we get that 

\begin{equation*}
    - \theta + \int_{\R^d} \left( \int_{\R^d} z \psi(x, \dr z) \right) \cdot b(t,x) \mu(\dr x) + \frac{1}{2} \int_{\R^d} \left|\int_{\R^d} z \psi(x, \dr z) \right|^2 - \mathcal{E}(t, \mu) \geq 0.
\end{equation*}

Using Jensen's inequality as well as the linearity of the integral yields the supersolution property. For the subsolution property, let us use the same arguments as before. This time however, we have to take extra-care to dependence in time of our Hamiltonian and to time-reversability of our transport equation. Let $(X,Y)$ be such that $\mathcal{L}(X,Y) = \gamma_o$, and define 

\begin{equation*}
\rho_s = \mathcal{L}\left( \left(1 - \lambda (t-s)\right)X + \lambda \left(t-s\right)Y - \int_s^t b(\tau, X) \dr \tau\right).
\end{equation*} 

If $z \in \text{Supp}(\rho_s)$ with $z = x - \lambda (t-s)(x-y) - \int_s^t b(\tau,x) \dr \tau $ for some $(x,y) \in \text{Supp}(\gamma_o)$, we want to show that such a decomposition is unique. Indeed, if

\begin{equation*}
    x - \lambda (t-s)(x-y) - \int_s^t b(\tau,x) \dr \tau  = x' - \lambda (t-s)(x'-y') - \int_s^t b(\tau,x') \dr \tau,
\end{equation*}

we get, using the fact that $b$ is Lipschitz,

\begin{equation}
    |x - x'| \leq  (t-s) \left( \lambda |x - y - (x' - y')| + C |x' -x| \right).
    \label{BOUNDSUPP}
\end{equation}

On the other hand,

\begin{equation*}
    x-y - (x' - y') =\frac{1}{\lambda(t-s)} \left( x' - x + \int_s^t b(\tau,x) - b(\tau,x') \dr \tau \right),
\end{equation*}

so, taking the scalar product with $x - y - (x' - y')$ and recalling that $(x - x') \cdot (y - y') \geq 0 $ as $(x,y),(x',y') \in \text{Supp}(\gamma_o)$,

\begin{equation*}
    |x - y - (x' - y')|^2 \leq \frac{1}{\lambda(t- s)} |x - x'|^2 + C|x - x'||x - y - (x' -y')|.
\end{equation*}

Let us denote $U = |x - y - (x' - y')|$ and $V = |x-x'|$. The previous inequality can be rewritten as

\begin{equation*}
    U^2 - C V U - \frac{V^2}{\lambda(t-s)} \leq 0.
\end{equation*}

The discriminant of the left-hand side is 

\begin{equation*}
    C^2 V^2 + \frac{4V^2}{\lambda(t-s)} = V^2 \left( C^2 + \frac{4}{\lambda(t-s)} \right) \geq 0,
\end{equation*}

which means that

\begin{equation*}
    U \leq \frac{CV}{2} + \sqrt{\frac{V^2}{\lambda(t-s)} + \frac{C^2 V^2}{4}} = \frac{V}{2} \left( C + \sqrt{\frac{4}{\lambda(t-s)} + C^2} \right).
\end{equation*}

From this we deduce

\begin{equation*}
    |x - y - (x' - y')| \leq \frac{|x-x'|}{2} \left( C + \sqrt{\frac{4}{\lambda(t-s)} + C^2} \right).
\end{equation*}

Plugging this back into \eqref{BOUNDSUPP} gives

\begin{equation*}
    |x - x'| \leq \left( \frac{3C}{2}(\lambda(t-s)) + \sqrt{4(\lambda(t-s)) + C^2 \lambda(t-s)^2} \right) |x-x'|.
\end{equation*}

which yields $x = x'$ for $s$ close enough to $t$, and then $y = y'$. As before, we can now define, for small $s$, $z = X_s(z) - \lambda(t-s) \left( X_s(z) - Y_s(z)\right) - \int_s^t b(\tau,X_s(z)) \dr \tau $, so that $\bm{\rho}$ solves the continuity equation with

\begin{equation*}
    v(s,z) = \lambda (X_s(z) - Y_s(z))  + b(s, X_s(z)).
\end{equation*}

We then get, from Bellman principle,

\begin{align*}
    \V(t, \mu) & \leq \int_{t- \delta}^t \int_{\R^d} \frac{1}{2} | \lambda(X_s(z) - Y_s(z)) + b(s, X_s(z)) - b(s,z)|^2 \dr \rho_s \dr s + \int_{t - \delta}^t \mathcal{E}(s, \rho_s) \dr s + \V(t-\delta, \rho_{t-\delta}) \notag \\
    & \leq \int_{t- \delta}^t \int_{\R^d} \frac{\lambda^2}{2} |X_s(z) - Y_s(z)|^2 + C\lambda|X_s(z) - Y_s(z)| |X_s(z) -z| + \frac{C^2}{2} |X_s(z) - z|^2 \dr \rho_s \dr s \notag \\
    & + \int_{t - \delta}^t \mathcal{E}(s, \rho_s) \dr s + \V(t-\delta, \rho_{t-\delta}).
\end{align*}

But 

\begin{equation*}
    |X_s(z) - z| \leq \lambda (t-s) |X_s(z) - Y_s(z)| + \int_{t-s}^t b(\tau, X_s(z)) \dr \tau,
\end{equation*}

so that 

\begin{equation*}
    \int_{t- \delta}^t \int_{\R^d} C\lambda|X_s(z) - Y_s(z)| |X_s(z) -z| + \frac{C^2}{2} |X_s(z) - z|^2 \dr \rho_s \dr s = o(\delta).
\end{equation*}

Now, using the fact that $(\theta, \psi) \in \partial^+ \V(t, \mu)$, we get

\begin{equation*}
     \V(t - \delta, \rho_{t- \delta}) \leq \V(t, \mu) - \delta \theta -  \mathbb{E} \left( \delta \lambda^2 |X-Y|^2 + \int_{t-\delta}^t b(\tau, X) \cdot (X-Y)\dr \tau  \right) + o(\delta),
\end{equation*}

so that 

\begin{equation*}
    0 \leq - \delta \theta - \frac{\lambda^2}{2} \mathbb{E} \left[ |X-Y|^2 \right] + \int_{t - \delta}^t \mathcal{E}(s, \rho_s) \dr s + o(\delta).
\end{equation*}

Dividing by $\delta$ and letting $\delta \rightarrow 0$ yields the correct inequality, as $\mathcal{E}$ is Lipschitz continuous.

\end{proof}

\section{A semi-Lagrangian scheme for a HJB equation on the Wasserstein space}
\label{NUMSCHEME}
\subsection{Description of the scheme}
In the finite dimensional Euclidean case, semi-Lagrangian schemes are well known schemes used for solving Hamilton-Jacobi-Bellman, see for instance \cite{FALCONE}. In the context of state estimation, this kind of scheme has been interpreted as a generalization of the Kalman filter to the nonlinear case in \cite{MOIREAU}. In this section, we try to adapt this approach to our setting. We consider a time discretization of \eqref{SIMPLE} given by 

\begin{equation*}
    \rho_{k+1} = (\text{id} + \tau ( b_k^n + v_k^n))_\# \rho_k, \\
\end{equation*}
with observations given by $\mathcal{E}_n(\rho_n)$ and where $b_k(\cdot) = b(\frac{kt}{n}, \cdot)$ . 

We begin by discretizing \eqref{CRITERION}, with $\tau = \frac{t}{n}$ for some $t \in \R_+$, through 

\begin{equation*}
    \mathcal{J}_n^- ( \bm{\rho} = (\rho_k)_{0 \leq k <n}) = \mathcal{V}_0(\rho_0) + \frac{1}{2\tau} \sum_{k=0}^{n-1} \W_2^2(\rho_{k+1}, (\id + \tau b_k^n)_\#\rho_k) + \tau \sum_{k=0}^{n-1} \mathcal{E}_k(\rho_k) ,
\end{equation*}

\begin{equation*}
        \mathcal{J}_n (\bm{\rho}) = \mathcal{V}_0(\rho_0) + \frac{1}{2\tau} \sum_{k=0}^{n-1} \W_2^2(\rho_{k+1},(\id + \tau b_k^n)_\#\rho_k) + \tau \sum_{k=0}^{n} \mathcal{E}_k(\rho_k). 
\end{equation*}

We also discretize the value function through 

\begin{equation*}
        \mathcal{V}_{n}^-(\rho) = \inf_{\substack{\bm{\rho} \in (\mathcal{P}_2(\mathbb{R}^d))^{n+1}  \\ \rho_n = \rho}} \left\{\mathcal{V}_0(\rho_0) + \frac{1}{2\tau} \sum_{k=0}^{n-1} \W_2^2(\rho_{k+1}, (\id + \tau b_k^n)_\#\rho_k) + \tau \sum_{k=0}^{n-1} \mathcal{E}_k(\rho_k)  \right\}.
\end{equation*}

\begin{equation*}
        \mathcal{V}_{n}(\rho) = \inf_{\substack{\bm{\rho} \in (\mathcal{P}_2(\mathbb{R}^d))^{n+1}  \\ \rho_n = \rho}} \left\{ \mathcal{V}_0(\rho_0) + \frac{1}{2\tau} \sum_{k=0}^{n-1} \mathcal{W}_2^2(\rho_{k+1}, (\id + \tau b_k^n)_\# \rho_k) + \tau \sum_{k=0}^{n} \mathcal{E}_k(\rho_k) \right\}.
\end{equation*}

We can then see that 

\begin{equation*}
\mathcal{V}_{n} (\rho) = \mathcal{V}_{n}^-(\rho) + \tau \mathcal{E}_n(\rho),
\end{equation*}

and similarly that 

\begin{equation*}
\mathcal{V}_{{n+1}}^-(\rho) = \inf_{\tilde{\rho} \in \mathcal{P}_2(\mathbb{R}^d)} \left\{ \mathcal{V}_{n}(\tilde{\rho}) + \frac{1}{2\tau} \mathcal{W}_2^2(\rho, (\id + \tau b_n^n)_\#\tilde{\rho}) \right\} .
\end{equation*}

These two step can be viewed as a generalization of the classical Kalman filter to the nonlinear and infinite dimensional case. Indeed, going from $\mathcal{V}_n^-$ to $\mathcal{V}_n$ can be seen as an update step, where we take into account the new observation $\mathcal{E}_n(\rho)$, while going from $\mathcal{V}_n$ to $\mathcal{V}_{n+1}^-$ can be seen as a prediction step, where we take into account the dynamics of our system. It mirrors the prediction-correction step of the original Kalman filter introduced in \cite{KALMAN} and generalized to the nonlinear setting in \cite{MOIREAU}.
We also have, for $n \in \mathbb{N},\, \rho \in \mathcal{P}_2(\mathbb{R}^d)$ 

\begin{equation*}
    \mathcal{V}_{n+1}^-(\rho) = \inf_{\tilde{\rho} \in \mathcal{P}_2(\mathbb{R}^d)} \left\{ \mathcal{V}_n^-(\tilde{\rho}) + \frac{1}{2 \tau} \mathcal{W}_2^2(\rho, (\id + \tau b_n^n)_\#\tilde{\rho}) + \tau \mathcal{E}_n(\tilde{\rho}) \right\},
\end{equation*}

\begin{equation*}
    \mathcal{V}_{n+1}(\rho) = \inf_{\tilde{\rho} \in \mathcal{P}_2(\mathbb{R}^d)} \left\{ \mathcal{V}_n(\tilde{\rho}) + \frac{1}{2 \tau} \mathcal{W}_2^2(\rho, (\id + \tau b_n^n)_\#\tilde{\rho}) + \tau \mathcal{E}_n(\rho) \right\}.
\end{equation*}

We will prove that this scheme is well defined by showing existence of minimizers in the definition of $\mathcal{V}_n$ and $\mathcal{V}_n^-$. We will also show that, as $n \rightarrow + \infty$, $\mathcal{V}_n$ converges to $\V(t, \cdot)$, where $\V$ is the value function defined in \eqref{HJBTRANS}, in the sense of $\Gamma$-convergence.
\subsection{Existence of minimizers}

We first prove that the scheme makes sense, by showing that the infimum in the definition of $\mathcal{V}_n$ is a minimum, under some assumptions on $\mathcal{E}_n$ and $\mathcal{V}_0$ that strongly resemble the assumptions of the continuous-in-time result we had in \autoref{MINI}.
\begin{theorem}\label{EXISTMIN}
Assume that
\begin{enumerate}
    \item $\V_0$ is l.s.c and non-negative,
    \item $\mu \in \Pro \rightarrow \mathcal{E}_k(\mu)$ is l.s.c. and bounded from below for every $0 \leq k \leq n$,
    \item $\mathcal{E}_k(\rho) < + \infty$ for every $\rho \in \Pro$.
\end{enumerate}
Then the infimum in the definition of $\mathcal{V}_n$ is a minimum.
\end{theorem}

In order to prove this result, we are going to use the following lemmas.    

\begin{lemma}
Under the assumptions of \autoref{EXISTMIN}, $\mathcal{J}_n$ is l.s.c. over $\left( \Pro \right)^{n+1}$ equipped with the product weak topology.
\end{lemma}

\begin{proof}
We have already shown that $(\mu, \nu) \longrightarrow \W_2^2(\mu,\nu)$ is l.s.c. on $\Pro \times \Pro$ in Lemma \ref{WLSC}. Since $b$ is smooth enough, we also get that $(\mu, \nu) \longrightarrow \W_2^2(\mu, (\id + \tau b_k^n)_\#\nu)$ is l.s.c. on $\Pro \times \Pro$  for every $0 \leq k \leq n-1$. 
Thanks to \cite[Proposition 7.1.]{OTAM}, we have that $\mathcal{V}_0(\rho_0)$ is l.s.c. Recalling that $\mathcal{E}_k$ is l.s.c. by assumption, the proof is complete.
\end{proof}

\begin{lemma}
Under the assumptions of \autoref{EXISTMIN}, for $t \in \mathbb{R}_+^* $ and $\rho \in \Pro$, the set 

\begin{equation*}
    \mathcal{A}_t = \left\{ \bm{\rho} \in \left( \Pro \right)^{n+1}, \mathcal{J}_n(\bm{\rho}) \leq t , \: \: \rho_n = \rho \right\}
\end{equation*}

is weakly compact.
\label{COMPACT}
\end{lemma}

\begin{proof}
Since $\mathcal{J}_n$ is weakly l.s.c. and since the projection on the last coordinate is continuous, $\mathcal{A}_t$ is weakly closed as one of the lower level set of a weakly l.s.c function. We then observe that, for $0 \leq k \leq n$, 

\begin{equation*}
   \W_2^2\left(\rho_k, \left((\id + \tau b_k^n)^{-1}\right)_\# \rho\right) = \W_2^2\left(\left(\id + \tau b_k^n\right)_\#\rho_k, \rho\right) \leq \sum_{j = k}^{n} \W_2^2((\id + \tau b_j^n)_\#\rho_j, \rho_{j+1}) \leq 2t \tau.
\end{equation*}

But, by the triangular inequality, 

\begin{equation*}
    \int_{\R^d} |x|^2 \,\mathrm{d}\rho_k \leq \W_2^2(((\id + \tau b_k^n))^{-1}_\#\rho, \rho) + \W_2^2(\rho_k, \rho) + \int_{\R^d} |x|^2 \,\mathrm{d}\rho.
\end{equation*}

Using the coupling between $\left((\id + \tau b_k^n)^{-1}\right)_\#\rho$ and $\rho$ defined through the transport map $T(x) = x + \tau b_k^n(x)$ yields

\begin{equation*}
    \W_2^2\left(\left((\id + \tau b_k^n)^{-1}\right)_\#\rho, \rho\right) \leq \int_{\R^d} \left| \tau b_k^n(x) \right|^2 \dr \rho \leq C
\end{equation*}

because $|b_k^n(x)|^2 \leq C(1 + |x|^2)$ since $b$ is Lipschitz. It means that $\rho_k$ has uniformly bounded second moments. It suffices, in order to conclude with Prokhorov's theorem (see for instance the proof in \cite[Theorem 1]{BALLS}).
\end{proof}

We are now able to prove the main result of this subsection.

\begin{proof}(of \autoref{EXISTMIN}) 

Because of our assumptions, $\mathcal{J}_n$ is proper. Let $\bm{\rho}_0 \in \left(\Pro \right)^{n+1}$ be such that 

\begin{equation*}
    t_0 = \mathcal{J}_n(\bm{\rho}_0) < + \infty.
\end{equation*}

We know with Lemma \ref{COMPACT} that $\mathcal{A}_{t_0}$ is weakly compact. We are thus minimizing a l.s.c. function, bounded from below, on a compact set. It thus admits a minimizer.

\end{proof}

\subsection{\texorpdfstring{$\Gamma$}{Gamma}-convergence of the scheme}

The goal of this section is to prove the $\Gamma$-convergence of the scheme towards the initial value function \eqref{VALUE}, of which we give the definition below. We are inspired mainly by two papers that were concerned with similar discretizations of the value function. In \cite{BENAMOUGAMMA}, the authors proved $\Gamma$-convergence of $\mathcal{J}_n$ but in the case of entropic optimal transport, whereas in \cite{LAVENANT}, they only proved the convergence of minimizers for the discrete problem towards minimizers of the continuous one. In both cases, convergence of the discretized value function was not shown. 

\begin{definition}
We say that $F_n : \Pro \rightarrow \bar{\R}$ $\Gamma$-converges towards $F$ if the following hold true:
\begin{itemize}
    \item For every $(\mu_n)_{n \in \mathbb{N}}$ such that $\mu_n \rightarrow \mu \in \Pro$ for the Wasserstein distance, we have
    \begin{equation*}
        F(\mu) \leq \liminf_{n \rightarrow + \infty} F_n(\mu_n).
    \end{equation*}
    
    \item For every $\mu \in \Pro$, there exists a sequence $(\mu_n)_{n \in \mathbb{N}}$ converging to $\mu$ such that
    \begin{equation*}
        F(\mu) \geq \limsup_{n \rightarrow + \infty} F_n(\mu_n).
    \end{equation*}
\end{itemize}
\end{definition}

By recalling that we set $\tau = \tau_n = \frac{t}{n}$ in our definition, we will prove the following result regarding the $\Gamma$-convergence of $\V_n$.

\begin{theorem}

Let $t \in \R_+$. Assume that:

\begin{itemize}
    \item $\V_0$ is l.s.c. and bounded below,
    \item $\mathcal{E}(t, \mu) \leq f(t) \psi\left( \int_{\R^d} |x|^2 \,\mathrm{d}\mu \right) + g(t)$ where $f, g, \psi$ are in $L_{loc}^\infty(\R_+)$,
    \item $t \rightarrow \mathcal{E}(t, \mu)$ is continuous for every $\mu \in \Pro$
    \item $\mu \in \Pro \rightarrow \mathcal{E}(t, \mu)$ is continuous with respect to the Wasserstein metric and bounded from below,
    \item for every $k \in \{0, \cdots, n \}, \mu \in \Pro$, $\mathcal{E}_k(\mu) = \mathcal{E}(\frac{kt}{n}, \mu)$. 
\end{itemize}
Then $\V_n$ $\Gamma$-converges towards $\V(t, \cdot)$.
\end{theorem}

\begin{proof}
Let us prove the inequality for the $\liminf$. Let $\mu^n \rightarrow \mu \in \Pro$ for the Wasserstein topology. Let $\bar{\bm{\rho}}^n = (\bar{\rho}_0^n, \cdots, \bar{\rho}_n^n)$ be a minimizer in the definition of $\V_n$ and assume that $\liminf_{n \rightarrow + \infty} \V_n(\mu_n) <+ \infty$, so that

\begin{equation*}
    \V_n(\mu^n) = \mathcal{V}_0(\bar{\rho}_0^n) + \frac{1}{2\tau_n} \sum_{k=0}^{n-1} \mathcal{W}_2^2(\bar{\rho}_{k+1}^n, (\id + \tau b_k^n)_\#\bar{\rho}_k^n) + \tau_n \sum_{k=0}^{n} \mathcal{E}_k(\bar{\rho}_k^n) 
\end{equation*}

Let us define $\rho^n(s) = \bar{\rho}_k^n$ for $s \in \left[\frac{k t}{n}, \frac{(k+1)t}{n} \right)$ and $\rho^n(t) = \mu^n$ so that $\bm{\rho}^n \in AC^2(0, t, \Pro)$. By construction, and with $b^n$ the piecewise interpolation of $b_k^n$,

\begin{equation*}
    \frac{1}{2\tau_n} \sum_{k=0}^{n-1} \mathcal{W}_2^2(\bar{\rho}_{k+1}^n, (\id + \tau b_k^n)_\#\bar{\rho}_k^n) \geq \frac{1}{2} \int_0^t \int_{\R^d} |v^n(s,x) - b^n(s,x)|^2 \dr \rho_s^n \,\mathrm{d}s.
\end{equation*}

Up to a subsequence, since $ \liminf_{n \rightarrow + \infty} \V_n(\mu_n) <+ \infty$, we can assume that $ \V_n(\mu_n)$ is bounded by a constant $C$. In particular, 

\begin{equation*}
    \frac{1}{2} \int_0^t \int_{\R^d} |v^n(s,x) - b_k^n(x)|^2 \dr \rho_s^n \,\mathrm{d}s \leq \V_n(\mu_n) \leq C.
\end{equation*}

By considering the curve $\tilde{\bm{\rho}}^n \in AC^2(0,t,\Pro)$ such that $\tilde{v}^n = v^n - b^n$, we are able to use \cite[Proposition 3 and 4]{GANGBOPOISSON} to get the existence of a limiting $AC^2$ curve $\bm{\rho}$ such that --~up to a subsequence~-- $\rho^n_s \rightarrow \rho_s$ for every $s \in [0,t]$ and that

\begin{equation*}
  \liminf_{n \rightarrow + \infty}   \frac{1}{2} \int_0^t \int_{\R^d} |v^n(s,x) - b_k^n(x)|^2 \dr \rho_s^n \,\mathrm{d}s \geq \frac{1}{2} \int_0^t \int_{\R^d} |v(s,x) - b(s,x)|^2 \dr \rho_s \,\mathrm{d}s.
\end{equation*}

This yields 

\begin{equation*}
    \liminf_{n \rightarrow + \infty}  \frac{1}{2\tau_n} \sum_{k=0}^{n-1} \mathcal{W}_2^2(\bar{\rho}_{k+1}^n, (\id + \tau b_k^n)_\# \bar{\rho}_k^n)  \geq \frac{1}{2} \int_0^t \int_{\R^d} |v(s,x) - b(s,x)|^2 \dr \rho_s \,\mathrm{d}s.
\end{equation*}

Using the convergence of $\rho^n$ at time $s=0$, we get in particular that $\bar{\rho}_0^n \rightarrow \rho(0)$. Together with the lower semi-continuity of $\mathcal{V}_0$, we also get 

\begin{equation*}
    \liminf_{n \rightarrow + \infty} \V_0 (\bar{\rho}_0^n) \geq \V_0 (\rho(0)).
\end{equation*}

Lastly, with the same notations as before, we have that

\begin{equation*}
    \tau_n \sum_{k=0}^{n} \mathcal{E}_k(\bar{\rho}_k^n) \geq \int_0^t \mathcal{E}(s, \rho^n(s)) \, \mathrm{d}s.
\end{equation*}

By Fatou's lemma (since $\mathcal{E}(s, \cdot)$ is l.s.c. and bounded from below), we then get

\begin{equation*}
     \liminf_{n \rightarrow + \infty} \sum_{k=0}^{n} \mathcal{E}_k(\bar{\rho}_k^n) \geq \int_0^t \mathcal{E}(s, \rho(s)) \, \mathrm{d}s.
\end{equation*}

In the end, we have obtained

\begin{equation*}
    \liminf_{n \rightarrow + \infty} \V_n(\mu^n) \geq \V_0 (\rho(0)) +\frac{1}{2} \int_0^t \int_{\R^d} |v(s,x) - b(s,x)|^2 \dr \rho_s \,\mathrm{d}s +  \int_0^t \mathcal{E}(s, \rho_s) \, \mathrm{d}s \geq \V(t, \mu),
\end{equation*}

which is exactly the first inequality needed for $\Gamma$-convergence.

\smallbreak

For the remaining inequality involving the $\limsup$, let us take $\mu \in \Pro$ and $\bm{\bar{\rho}}$ a minimizer in the definition of $\V$ so that

\begin{equation*}
    \V(t, \mu) = \V_0 (\bar{\rho}(0)) + \frac{1}{2} \int_0^t \int_{\R^d} |\bar{v}(s,x) - b(s,x)|^2 \dr \bar{\rho}_s \,\mathrm{d}s +  \int_0^t \mathcal{E}(s, \bar{\rho}(s)) \, \mathrm{d}s.
\end{equation*}

Let us define $\rho^n_k = \bar{\rho}(\frac{kt}{n})$ for every $k \in \{0, \cdots, n \}$. In particular, $\V_0(\rho_0^n) = \V_0(\bar{\rho}(0))$. We get that

\begin{equation*}
    \frac{1}{2\tau} \W_2^2((\id + \tau b_k^n)_\#\rho_k^n, \rho_{k+1}^n) \leq \frac{1}{2} \int_\frac{kt}{n}^\frac{(k+1)t}{n} \int_{\R^d} |\bar{v}(s,x) - b_k^n(x)|^2 \dr \bar{\rho}_s \,\mathrm{d}s.
\end{equation*}

Summing over $k$, we get

\begin{equation*}
    \frac{1}{2 \tau_n} \sum_{k= 0}^{n-1} \W_2^2((\id + \tau b_k^n)_\# \rho_k^n, \rho_{k+1}^n) \leq \frac{1}{2} \int_0^t\int_{\R^d} |\bar{v}(s,x) - b_k^n(x)|^2 \dr \bar{\rho}_s \,\mathrm{d}s,
\end{equation*}
 hence
 
 \begin{equation*}
    \limsup_{n \rightarrow + \infty} \frac{1}{2 \tau_n} \sum_{k= 0}^{n-1} \W_2^2((\id + \tau b_k^n)_\# \rho_k^n, \rho_{k+1}^n) \leq \frac{1}{2} \int_0^t \int_{\R^d} |\bar{v}(s,x) - b(s,x)|^2 \dr \bar{\rho}_s \,\mathrm{d}s.
 \end{equation*}
 
 Lastly, we handle the remaining term through
 
 \begin{equation*}
     \tau_n \sum_{k= 0}^{n-1} \mathcal{E}_k(\rho_k^n) = \int_0^t \mathcal{E}(s, \rho^n(s)) \, \mathrm{d}s,
 \end{equation*}
 
 where $\rho^n = \rho_k^n$ for $s \in \left[ \frac{k t}{n}, \frac{(k+1)t}{n} \right)$. But, if $s \in [0,t]$,
 
 \begin{equation*}
     \mathcal{E}(s, \rho^n(s)) \leq  f(s) \psi\left( \int_{\R^d} |x|^2 \,\mathrm{d}\rho^n(s) \right) + g(s).
 \end{equation*}
 
 But $\rho^n(s) \rightarrow \bar{\rho}(s)$ by construction, in particular the sequence of its second moment converges and is hence bounded, which means that $\mathcal{E}(s, \rho^n(s))$ is bounded uniformly in $n$ and $s$ (since $f,g$ and $\psi$ are in $L^\infty_{\text{loc}}$). We can then apply reverse Fatou's lemma and use the upper semi-continuity of $\mathcal{E}$ in order to get
 
 \begin{equation*}
      \limsup_{n \rightarrow + \infty} \int_0^t \mathcal{E}(s, \rho^n(s)) \, \mathrm{d}s \leq \int_0^t \limsup_{n \rightarrow + \infty} \mathcal{E}(s, \rho^n(s)) \, \mathrm{d}s \leq \int_0^t  \mathcal{E}(s, \bar{\rho}(s)) \, \mathrm{d}s.
 \end{equation*}
 
 This yields
 
 \begin{equation*}
     \limsup_{n \rightarrow + \infty} \V_n(\mu^n) \leq \limsup_{n \rightarrow + \infty} \V_0(\rho_0^n) + \frac{1}{2 \tau_n} \sum_{k= 0}^{n-1} \W_2^2(\rho_{k+1}^n, (\id + \tau b_k^n)_\# \rho_{k}^n) +  \tau_n \sum_{k= 0}^{n-1} \mathcal{E}_k(\rho_k^n) \leq \V(t, \mu).
 \end{equation*}
 \end{proof}

 \section{Application to deterministic filtering and generalization of the Mortensen observer} \label{LINK}

As previously stated, our construction aims to generalize the Mortensen observer in the Wasserstein space. However, we would like to highlight how our value function is in fact tied to the classical Mortensen observer introduced in \autoref{MORTENSENCLASSIC}, and how it can be seen as its generalization for ODEs with a random initial condition. Indeed, let us consider $V$ a classical $C^1$ solution of the HJB equation

\begin{equation} \label{CLASSICHJB}
    \begin{dcases}
        \partial_t V (t,x) + b(t,x) \cdot \nabla V(t,x) + \frac{1}{2} \left| \nabla V (t,x) \right|^2 = e(t,x), \\
        V(0,x) = V_0(x).    
    \end{dcases}
\end{equation}

Further assume that

\begin{equation} \label{QUADRESTIMATE}
    V(t,x) + e(t,x) \leq C\left( 1 + |x|^2 \right),
\end{equation}

so that $\mathcal{V}(t, \mu) = \int_{\R^d} V(t,x) \mu (\dr x)$ and $\mathcal{E}(t, \mu) = \int_{\R^d} e(t,x) \mu(\dr x)$ are well-defined for $\mu \in \Pro$.

\begin{remark}
    We recall that \eqref{QUADRESTIMATE} holds under the assumptions given in \autoref{MORTENSENCLASSIC}.         
\end{remark}

We then have that 

\begin{equation*}
    \begin{dcases}
        \partial_t \V (t,\mu) + \langle b(t, \cdot ), \nabla \V (t,\mu) \rangle_{\mathrm{L}^2(\mu)}+ \frac{1}{2} \left\| \nabla \V (t,\mu) \right\|_{\mathrm{L}^2(\mu)}^2 = \mathcal{E}(t,\mu), \\
        \V(0,\mu) = \int_{\R^d} V_0(x) \mu(\dr x).    
    \end{dcases}
\end{equation*}

By the comparison principle, $\V$ coincides with the value function of the associated optimal control problem. In fact, this result holds for a more general class of solutions.

\begin{theorem} 
    Let $V$ be a viscosity solution to \eqref{CLASSICHJB} such that:
    \begin{itemize}
        \item $V$ is semiconcave in $x$, uniformly in time,
        \item the condition \eqref{QUADRESTIMATE} holds.
    \end{itemize}
    Then $\V(t, \mu) = \int_{\R^d} V(t,x) \mu(\dr x)$ coincides with the value function 

\begin{equation*}
    \mathcal{U}(t,\rho) = \inf_{\substack{\bm{\rho} \in AC^2(0,t, \Pro) \\ \rho_t = \rho}} \left\{ \int_{\R^d} V_0(x) \rho_0(\dr x) + \int_0^t \int_{\mathbb{R}^d} \frac{1}{2} |v(s,x) - b(s,x)|^2 \,\mathrm{d}\rho_s \,\mathrm{d}s + \int_0^t \mathcal{E}(s, \rho_s) \,\mathrm{d}s\right\}.
\end{equation*}

\end{theorem}

\begin{proof}
    Since $V$ is a viscosity solution to \eqref{CLASSICHJB}, for every $(t,x) \in \R_+ \times \R^d$,  $(\tau, p) \in \partial^+ V(t,x)$, 
    \begin{equation*}
        \tau + \frac{1}{2} | p |^2 - e(t,x) \leq 0.
    \end{equation*}

    Because $e(t, \cdot) \in \mathrm{L}^2(\mu)$ for every $\mu \in \Pro$, so does $\tau$. We fix $\mu \in \Pro$. Let us then define $\theta = \int_{\R^d} \tau \mu$ and take $\psi(x) \in \mathcal{P}(\R^d)$ such that $\left\{ \tau \right\} \times \text{Supp}(\psi(x)) \subset \partial^+ V(t,x)$. Then, 

    \begin{equation*}
        \theta + \frac{1}{2} \int_{\R^d \times \R^d} |p|^2 \psi(x, \dr p) \mu(\dr x) - \int_{\R^d} e(t,x) \mu(\dr x) \leq 0.
    \end{equation*}

    But we want this inequality to hold for any $(\theta, \psi) \in \partial^+ \V(t, \mu)$. We thus show that the superdifferential of $\V$ consists of elements of the form
    \begin{equation*}
        \left\{ (\theta, \psi), \, \theta = \int_{\R^d} \tau \mu, \quad \{ \tau \} \times\text{Supp}(\psi(x)) \subset \partial^+ V(t,x)  \right\}.
    \end{equation*}

    First of all, if $(\theta, \psi) \in \partial^+\V(t,\mu)$, by using the definition of the superdifferential with Dirac masses, one gets that

    \begin{equation*}
        V(s,y) - V(t,x) \leq \tau (s-t) + \langle p, y - x \rangle + o(|x-y| + |s-t|),
    \end{equation*}

    since the only coupling between $\delta_x$ and $\delta_y$ is $\delta_{(x,y)}$. Here, $p = \int_{\R^d} z \psi(x, \dr z)$. In particular, $(\tau, p) \in \partial^+ V(t,x)$. For the reverse inclusion, since $V$ is semi-concave uniformly in time, there exists $C>0$ such that, for every $(\tau, p) \in \partial^+V(t, x)$,

    \begin{equation*}
        V(s,y) \leq V(t,x) + \langle p, y-x \rangle + \tau(s-t)+  C |x - y|^2 + o(|t-s|).
    \end{equation*}

    Let $\mu, \nu \in \Pro$, $\gamma \in \Pi(\mu, \nu)$ and $\psi$ such that $\{ \tau \} \times\text{Supp}(\psi(x)) \subset \partial^+ V(t,x)$. Integrating w.r.t. $\psi(x, \dr p) \gamma(\dr x, \dr y)$ yields

    \begin{equation*}
        \V(s, \nu) \leq \V(t, \mu) + \int_{\R^d \times \R^d \times \R^d} \langle p, y-x \rangle \psi(x, \dr p) \gamma(\dr x, \dr y) + C \int_{\R^d \times \R^d} |x-y|^2 \gamma(\dr x, \dr y) + o(|t-s|).
    \end{equation*}

    This concludes the proof, as $C \int_{\R^d \times \R^d} |x-y|^2 \gamma(\dr x, \dr y) = o\left( \left( \int_{\R^d \times \R^d} |x-y|^2 \gamma(\dr x, \dr y) \right)^\frac{1}{2} \right)$.
\end{proof}

In particular, the set of minimizers of $\V$ is exactly the measures concentrated on the argmin of $V$, i.e.

\begin{equation*}
    \argmin_{\mu \in \Pro} \V(t, \mu) = \left\{ \mu \in \Pro, \quad \text{Supp}(\mu) \subset \argmin_{x \in \R^d} V(t, x) \right\}.
\end{equation*}

In particular, it allows us to track several minima of the initial value function. For instance, if $V(t, \cdot)$ has $N(t)$ minima $x_1(t), \cdots, x_{N(t)}(t)$, then

\begin{equation*}
    \hat{\mu}(t) = \sum_{i = 1}^{N(t)} \delta_{x_i(t)} \in \argmin \V(t, \cdot).  
\end{equation*}

In fact, we interpret this example as a way of generalizing the Mortensen observer to systems of the form
\begin{equation*}
    \begin{dcases}
    \dot{x}(t) = b(t,x(t)), \\
    x(0) \sim \mu_0,
    \end{dcases}
\end{equation*}

i.e. systems of ODEs with random initial condition. This problem has been thoroughly studied in \cite{JEAN, JIMENEZ} for instance, from an optimal control point of view. Our results provide an estimation counterpart to the result presented in those works. However, we want to stress that this construction is not as robust as it may seem at first glance. Indeed, let us take our previous example of a value function with several minima. Then both 

\begin{equation*}
    \hat{\mu}_1(t) = \delta_{x_1(t)} \: \text{ and } \: \hat{\mu}(t) = \sum_{i = 1}^{N(t)} \delta_{x_i(t)},
\end{equation*}
 
are minima of $\V$. We believe that both of these behaviours can be expected, depending on the properties of $V_0$ and $f$. It shows that this straightforward approach may fail to spot several minimizers. In order to do so, one may try to add some criterion in order to select the uniform law on $\argmin V(t,\cdot)$, i.e. to select every minimizer with equal weight instead of potentially selecting one and only one of the potential minimizers. For instance, if we assume that $\argmin V(t,\cdot) \subset K_t$ with $K_t \subset \R^d$ some compact, we can consider 

\begin{multline*}
    \tilde{\V}(t,\rho) = \inf_{\substack{\bm{\rho} \in AC^2(0,t, \Pro) \\ \rho_t = \rho}} \left\{ \int_{\R^d} V_0(x) \rho_0(\dr x) + \int_0^t \int_{\mathbb{R}^d} \frac{1}{2} |v(s,x) - b(s,x)|^2 \,\mathrm{d}\rho_s \,\mathrm{d}s \right.\\ 
    \left. + \int_0^t \int_{\R^d} f(s,x) \rho_s(\dr x) \,\mathrm{d}s + \int_0^t \W_2(\rho_s, \mathcal{U}(K_t)) \dr s \right\},
\end{multline*}

where $\mathcal{U}(K_t)$ is the uniform law on $K_t$. We believe that such a construction can be a good approximation of the uniform law on $\argmin V(t,\cdot)$, which highlights the need to consider more general HJB equations on the Wasserstein space. 

\section{Conclusion and perspectives}

In this work, we developed a deterministic filtering framework posed directly on the Wasserstein space $\mathcal P_2(\mathbb R^d)$, motivated by the minimum-energy formulation of the Mortensen observer. By interpreting density reconstruction as an optimal transport problem with observations, we defined a value function combining a kinetic transport cost in the sense of Benamou-Brenier and a time-dependent observation functional. This led naturally to a Hamilton-Jacobi-Bellman equation on $\mathcal P_2(\mathbb R^d)$ involving the Wasserstein gradient.

We established the main analytical properties of the resulting value function. In particular, we proved a dynamic programming principle, continuity with respect to time and the Wasserstein topology, and existence of minimizing trajectories. Under suitable regularity and growth assumptions on the observation functional, we showed that the value function is a viscosity solution of the associated Hamilton-Jacobi equation. Two complementary notions of viscosity solutions were developed: an intrinsic formulation based on Wasserstein subdifferentials and a Hilbertian formulation inspired by the lifting approach of Lions. The latter allowed us to prove a comparison principle and uniqueness of solutions. We also demonstrated how the framework extends to transport equations with drift, thereby covering a broader class of dynamics.

Beyond its theoretical contributions, this work provides a conceptual bridge between deterministic filtering, optimal transport, and mean-field control. From the perspective of filtering theory, it offers a natural extension of the Mortensen observer to density-valued states. From the viewpoint of optimal transport, it gives a control-theoretic interpretation of Hamilton-Jacobi equations on Wasserstein space with source terms arising from observations.

Several directions for future research emerge from this study. A first perspective concerns numerical approximation. Designing structure-preserving algorithms for the proposed filtering problem, possibly relying on entropic regularization or particle methods consistent with the underlying Wasserstein geometry, remains an open challenge. Finally, it would be of interest to apply the proposed framework to concrete models arising in collective dynamics (for instance involving a dissipative term), where observations are naturally expressed as functionals of probability distributions. We believe that the present work lays the groundwork for a systematic theory of observer design and filtering on Wasserstein spaces and opens new perspectives at the interface of control theory, partial differential equations, and optimal transport.

\section*{Acknowledgments}

The author would like to thank Charles Bertucci for introducing him to the concept of viscosity solutions, Louis-Pierre Chaintron for his comments on the transport equation and its link to (stochastic) filtering, and Philippe Moireau for introducing him to the problem of deterministic filtering, for helpful discussions on this subject and for his valuable feedback on the first version of this work.

\printbibliography
\end{document}